\makeatletter

\documentclass[
    11pt,
    a4paper,
    oneside,
    openright,
    center,
    chapterbib,
    crosshair,
    fleqn,
    headcount,
    headline,
    indent,
    indentfirst=false,
    portrait,
    phonetic,
    oldernstyle,
    onecolumn,
    sfbold,
    upper,
]{amsart}

\let\th@plain\relax

\PassOptionsToPackage{T2A}{fontenc} 
\PassOptionsToPackage{utf8}{inputenc} 
\PassOptionsToPackage{cyrpart}{cyrillic}
\PassOptionsToPackage{british,ngerman,russian}{babel}
\PassOptionsToPackage{
    english,
    ngerman,
    russian,
    capitalise,
}{cleveref}
\PassOptionsToPackage{
    margin=10pt,
    position=bottom,
    justification=centering,
    font=small,
    labelformat=simple,
    labelfont={sc},
    labelsep={period},
    textfont={it}
}{caption}
\PassOptionsToPackage{
    margin=10pt,
    position=bottom,
    justification=centering,
    font=small,
    labelformat=parens,
    labelfont={bf},
    labelsep={space},
    textfont={it}
}{subcaption}
\PassOptionsToPackage{framemethod=TikZ}{mdframed}
\PassOptionsToPackage{
    bookmarks=true,
    bookmarksopen=false,
    bookmarksopenlevel=0,
    bookmarkstype=toc,
    colorlinks=false,
    raiselinks=true,
    hyperfigures=true,
}{hyperref}
\PassOptionsToPackage{normalem}{ulem}
\PassOptionsToPackage{
    amsmath,
    thmmarks,
}{ntheorem}
\PassOptionsToPackage{table}{xcolor}
\PassOptionsToPackage{
    all,
    color,
    curve,
    frame,
    import,
    knot,
    line,
    movie,
    rotate,
    textures,
    tile,
    tips,
    web,
    xdvi,
}{xy}
\PassOptionsToPackage{
    reset,
    left=1in,
    right=1in,
    top=20mm,
    bottom=20mm,
    heightrounded,
}{geometry}
\PassOptionsToPackage{
    symbol*, 
    multiple, 
}{footmisc}
\PassOptionsToPackage{overload}{textcase}
\PassOptionsToPackage{indentafter}{titlesec}

\usepackage{amsfonts}
\usepackage{amsmath}
\usepackage{amssymb}
\usepackage{ntheorem} 
\usepackage{array}
\usepackage{babel}
\usepackage{bbding}
\usepackage{bbm}
\usepackage{bibentry}
\usepackage{booktabs}
\usepackage{bold-extra} 
\usepackage{calc}
\usepackage{cancel}
\usepackage{caption} 
\usepackage{changepage}
\usepackage{cjhebrew}
\usepackage{cmlgc}
\usepackage{colonequals}
\usepackage{color}
\usepackage{comment}
\usepackage{datetime}
\usepackage{dsfont}
\usepackage{etex}
\usepackage{etoolbox}
\usepackage{eurosym}
\usepackage{fancybox}
\usepackage{fancyhdr}
\usepackage{float}
\usepackage{fontenc}
\usepackage{footmisc}
\usepackage{fp}
\usepackage{geometry}
\usepackage{graphicx}
\usepackage{ifpdf}
\usepackage{ifthen}
\usepackage{ifoddpage}
\usepackage{ifnextok}  
\usepackage{index}     
\usepackage{inputenc}
\usepackage{latexsym}
\usepackage{lineno}
\usepackage{listings}
\usepackage{longtable}
\usepackage{lscape}
\usepackage{mathrsfs}
\usepackage{multicol}
\usepackage{multirow}
\usepackage{nameref}
\usepackage{nowtoaux}
\usepackage{paralist}
\usepackage{enumerate} 
\usepackage{pgf}
\usepackage{pgfplots}
\usepackage{phonetic}
\usepackage{proof}
\usepackage{qtree}
\usepackage{refcount}
\usepackage{savesym}
\usepackage{stmaryrd}
\usepackage{synttree}
\usepackage{subcaption}
\usepackage{suffix}
\usepackage{yfonts} 
\usepackage{textcase} 
\usepackage{tikz}
\usepackage{xy}
\usepackage{wrapfig}
\usepackage{xcolor}
\usepackage{xspace}
\usepackage{xstring}
\usepackage{hyperref}
\usepackage{arydshln}
\usepackage{cleveref} 

\pgfplotsset{compat=newest}
\usetikzlibrary{math}

\usetikzlibrary{
    angles,
    arrows,
    automata,
    calc,
    decorations,
    decorations.pathmorphing,
    decorations.pathreplacing,
    positioning,
    patterns,
    quotes,
}

\savesymbol{Diamond}
\savesymbol{emptyset}
\savesymbol{ggg}
\savesymbol{int}
\savesymbol{lll}
\savesymbol{RectangleBold}
\savesymbol{langle}
\savesymbol{rangle}
\savesymbol{hookrightarrow}
\savesymbol{hookleftarrow}
\savesymbol{Asterisk}
\usepackage{mathabx}
\usepackage{wasysym}

\restoresymbol{x}{Diamond}
\restoresymbol{x}{emptyset}
\restoresymbol{x}{ggg}
\restoresymbol{x}{int}
\restoresymbol{x}{lll}
\restoresymbol{x}{RectangleBold}
\restoresymbol{x}{langle}
\restoresymbol{x}{rangle}
\restoresymbol{x}{hookrightarrow}
\restoresymbol{x}{hookleftarrow}
\restoresymbol{x}{Asterisk}

\ifpdf
    \usepackage{pdfcolmk}
\fi

\usepackage{mdframed}

\DeclareFontFamily{U}{MnSymbolA}{}
\DeclareFontShape{U}{MnSymbolA}{m}{n}{
    <-6> MnSymbolA5
    <6-7> MnSymbolA6
    <7-8> MnSymbolA7
    <8-9> MnSymbolA8
    <9-10> MnSymbolA9
    <10-12> MnSymbolA10
    <12-> MnSymbolA12
}{}
\DeclareFontShape{U}{MnSymbolA}{b}{n}{
    <-6> MnSymbolA-Bold5
    <6-7> MnSymbolA-Bold6
    <7-8> MnSymbolA-Bold7
    <8-9> MnSymbolA-Bold8
    <9-10> MnSymbolA-Bold9
    <10-12> MnSymbolA-Bold10
    <12-> MnSymbolA-Bold12
}{}
\DeclareSymbolFont{MnSyA}{U}{MnSymbolA}{m}{n}
\DeclareMathSymbol{\lcirclearrowright}{\mathrel}{MnSyA}{252}
\DeclareMathSymbol{\lcirclearrowdown}{\mathrel}{MnSyA}{255}
\DeclareMathSymbol{\rcirclearrowleft}{\mathrel}{MnSyA}{250}
\DeclareMathSymbol{\rcirclearrowdown}{\mathrel}{MnSyA}{251}

\DeclareFontFamily{U}{MnSymbolC}{}
\DeclareSymbolFont{MnSyC}{U}{MnSymbolC}{m}{n}
\DeclareFontShape{U}{MnSymbolC}{m}{n}{
    <-6>  MnSymbolC5
    <6-7>  MnSymbolC6
    <7-8>  MnSymbolC7
    <8-9>  MnSymbolC8
    <9-10> MnSymbolC9
    <10-12> MnSymbolC10
    <12->   MnSymbolC12%
}{}
\DeclareMathSymbol{\powerset}{\mathord}{MnSyC}{180}
\DeclareMathSymbol{\righthalfcap}{\mathbin}{MnSyC}{186}

\DeclareMathAlphabet{\mathpzc}{OT1}{pzc}{m}{it}

\def\boolwahr{true}
\def\boolfalsch{false}
\def\boolleer{}

\let\boolinappendix\boolfalsch
\let\boolinmdframed\boolfalsch
\let\eqtagset\boolfalsch
\let\eqtaglabel\boolleer
\let\eqtagsymb\boolleer

\newcount\bufferctr
\newcount\bufferreplace

\newlength\rtab
\newlength\gesamtlinkerRand
\newlength\gesamtrechterRand
\newlength\ownspaceabovethm
\newlength\ownspacebelowthm
\def\secnumberingpt{.}
\def\secnumberingseppt{.}
\def\subsecnumberingseppt{}
\def\thmnumberingpt{.}
\def\thmnumberingseppt{}
\def\thmForceSepPt{.}

\definecolor{leer}{gray}{1}
\definecolor{boxgrau}{gray}{0.85}
\definecolor{dunkelgrau}{gray}{0.5}
\definecolor{maroon}{rgb}{0.6901961,0.1882353,0.3764706}
\definecolor{dunkelgruen}{rgb}{0.015625,0.363281,0.109375}
\definecolor{dunkelrot}{rgb}{0.5450980392,0,0}
\definecolor{dunkelblau}{rgb}{0,0,0.5450980392}
\definecolor{blau}{rgb}{0,0,1}
\definecolor{newresult}{rgb}{0.6,0.6,0.6}
\definecolor{improvedresult}{rgb}{0.9,0.9,0.9}
\definecolor{hervorheben}{rgb}{0,0.9,0.7}
\definecolor{starkesblau}{rgb}{0.1019607843,0.3176470588,0.8156862745}
\definecolor{achtung}{rgb}{1,0.5,0.5}
\definecolor{frage}{rgb}{0.5,1,0.5}
\definecolor{schreibweise}{rgb}{0,0.7,0.9}
\definecolor{axiom}{rgb}{0,0.3,0.3}

\def\let@name#1#2{
    \expandafter\let\csname #1\expandafter\endcsname\csname #2\endcsname\relax
}
\DeclareRobustCommand\crfamily{\fontfamily{ccr}\selectfont}
\DeclareTextFontCommand{\textcr}{\crfamily}

\def\ifthenelseleer#1#2#3{\ifthenelse{\equal{#1}{}}{#2}{#1#3}}
\def\bedingtesspaceexpand#1#2#3{\ifthenelseleer{\csname #1\endcsname}{#3}{#2#3}}

\def\nvraum{\@ifnextchar\bgroup{\nvraum@c}{\nvraum@bes}}
    \def\nvraum@c#1{\vspace*{-#1\baselineskip}}
    \def\nvraum@bes{\vspace*{-\baselineskip}}
\def\erlaubeplatz{\relax\ifmmode\else\@\xspace\fi}
\def\entferneplatz{\relax\ifmmode\else\expandafter\@gobble\fi}

\def\send@toaux#1{\@bsphack\protected@write\@auxout{}{\string#1}\@esphack}

\def\rlabel#1[#2]#3#4#5{#5\rlabel@aux{#1}[#2]{#3}{#4}{#5}}
    \def\rlabel@aux#1[#2]#3#4#5{%
        \send@toaux{\newlabel{#1}{{\@currentlabel}{\thepage}{{\unexpanded{#5}}}{#2.\csname the#2\endcsname}{}}}\relax%
    }

\def\tag@rawscheme#1#2[#3]#4#5{\@ifnextchar[{\tag@rawscheme@{#1}{#2}[#3]{#4}{#5}}{\tag@rawscheme@{#1}{#2}[#3]{#4}{#5}[*]}}
    \def\tag@rawscheme@#1#2[#3]#4#5[#6]{\@ifnextchar\bgroup{\tag@rawscheme@@{#1}{#2}[#3]{#4}{#5}[#6]}{\tag@rawscheme@@{#1}{#2}[#3]{#4}{#5}[#6]{}}}
    \def\tag@rawscheme@@#1#2[#3]#4#5[#6]#7{%
        \ifthenelse{\equal{#6}{*}}{%
            \ifthenelse{\equal{#7}{\boolleer}}{\refstepcounter{#3}#4\csname the#3\endcsname#5}{#4#7#5}%
        }{%
            \refstepcounter{#3}#4%
            \ifthenelse{\equal{#7}{\boolleer}}{\rlabel{#6}[#3]{#1}{#2}{\csname the#3\endcsname}}{\rlabel{#6}[#3]{#1}{#2}{#7}}%
            #5%
        }%
    }
\def\tag@scheme#1#2[#3]{\tag@rawscheme{#1}{#2}[#3]{\upshape(}{\upshape)}}

\def\eqtag@post#1{\makebox[0pt][r]{#1}}
\def\eqtag@pre{\tag@scheme{Eq}{Equation}[Xe]}
\def\eqtag{\@ifnextchar[{\eqtag@}{\eqtag@[*]}}
    \def\eqtag@[#1]{\@ifnextchar\bgroup{\eqtag@@[#1]}{\eqtag@@[#1]{}}}
    \def\eqtag@@[#1]#2{\eqtag@post{\eqtag@pre[#1]{#2}}}

\def\eqcref#1{\text{(\ref{#1})}}

\def\punktcref#1{\eqcref{it:#1:\beweislabel}}

\def\opfromto[#1]_#2^#3{\underset{#2}{\overset{#3}{#1}}}
\def\textoverset#1#2{\overset{\text{#1}}{#2}}

\def\eqcrefoverset#1#2{\textoverset{\eqcref{#1}}{#2}}

\def\mathclap#1{#1}
\def\oberunterset#1{\@ifnextchar^{\oberunterset@oben{#1}}{\oberunterset@unten{#1}}}
    \def\oberunterset@oben#1^#2_#3{\underset{\mathclap{#3}}{\overset{\mathclap{#2}}{#1}}}
    \def\oberunterset@unten#1_#2^#3{\underset{\mathclap{#2}}{\overset{\mathclap{#3}}{#1}}}
    \def\breitunderbrace#1_#2{\underbrace{#1}_{\mathclap{#2}}}
    \def\breitoverbrace#1^#2{\overbrace{#1}^{\mathclap{#2}}}
    \def\breitunderbracket#1_#2{\underbracket{#1}_{\mathclap{#2}}}
    \def\breitoverbracket#1^#2{\overbracket{#1}^{\mathclap{#2}}}

\def\generatenestedsecnumbering#1#2#3{%
    \expandafter\gdef\csname thelong#3\endcsname{%
        \expandafter\csname the#2\endcsname%
        \secnumberingpt%
        \expandafter\csname #1\endcsname{#3}%
    }%
    \expandafter\gdef\csname theshort#3\endcsname{%
        \expandafter\csname #1\endcsname{#3}%
    }%
}
\def\generatenestedthmnumbering#1#2#3{%
    \expandafter\gdef\csname the#3\endcsname{%
        \expandafter\csname the#2\endcsname%
        \thmnumberingpt%
        \expandafter\csname #1\endcsname{#3}%
    }%
    \expandafter\gdef\csname theshort#3\endcsname{%
        \expandafter\csname #1\endcsname{#3}%
    }%
}

\def\+#1{\addtocounter{#1}{1}}
\def\setcounternach#1#2{\setcounter{#1}{#2}\addtocounter{#1}{-1}}

\def\forcepunkt#1{#1\IfEndWith{#1}{.}{}{.}}
\def\lateinabkuerzung#1#2{%
    \expandafter\gdef\csname #1\endcsname{\emph{#2}\@ifnextchar.{\entferneplatz}{\erlaubeplatz}}
}
\def\deutscheabkuerzung#1#2{%
    \expandafter\gdef\csname #1\endcsname{{#2}\@ifnextchar.{\entferneplatz}{\erlaubeplatz}}
}

\def\matrix#1{\left(\begin{array}{#1}}
    \def\endmatrix{\end{array}\right)}
\def\smatrix{\left(\begin{smallmatrix}}
    \def\endsmatrix{\end{smallmatrix}\right)}

\def\multiargrekursiverbefehl#1#2#3#4#5#6#7#8{%
    \expandafter\gdef\csname#1\endcsname #2##1#4{\csname #1@anfang\endcsname##1#3\egroup}
    \expandafter\def\csname #1@anfang\endcsname##1#3{#5##1\@ifnextchar\egroup{\csname #1@ende\endcsname}{#7\csname #1@mitte\endcsname}}
    \expandafter\def\csname #1@mitte\endcsname##1#3{#6##1\@ifnextchar\egroup{\csname #1@ende\endcsname}{#7\csname #1@mitte\endcsname}}
    \expandafter\def\csname #1@ende\endcsname##1{#8}
}
\multiargrekursiverbefehl{svektor}{[}{;}{]}{\begin{smatrix}}{}{\\}{\\\end{smatrix}}
\multiargrekursiverbefehl{vektor}{[}{;}{]}{\begin{matrix}{c}}{}{\\}{\\\end{matrix}}
\multiargrekursiverbefehl{vektorzeile}{}{,}{;}{}{&}{}{}
\multiargrekursiverbefehl{matlabmatrix}{[}{;}{]}{\begin{smatrix}\vektorzeile}{\vektorzeile}{;\\}{;\end{smatrix}}

\def\faelle[#1]#2{\left\{\begin{array}[#1]{#2}}
    \def\endfaelle{\end{array}\right.}

\def\BeweisRichtung[#1]{\@ifnextchar\bgroup{\@BeweisRichtung@c[#1]}{\@BeweisRichtung@bes[#1]}}
    \def\@BeweisRichtung@bes[#1]{{\bfseries(#1).~}}
    \def\@BeweisRichtung@c[#1]#2#3{{\bfseries(#2 #1 #3).~}}
\def\erzeugeBeweisRichtungBefehle#1#2{
    \expandafter\gdef\csname #1text\endcsname##1##2{\BeweisRichtung[#2]{##1}{##2}}
    \expandafter\gdef\csname #1\endcsname{%
        \@ifnextchar\bgroup{\csname #1@\endcsname}{\csname #1text\endcsname{}{}}%
    }
    \expandafter\gdef\csname #1@\endcsname##1##2{%
        \csname #1text\endcsname{\punktcref{##1}}{\punktcref{##2}}%
    }
}
\erzeugeBeweisRichtungBefehle{hinRichtung}{$\Rightarrow$}
\erzeugeBeweisRichtungBefehle{herRichtung}{$\Leftarrow$}
\erzeugeBeweisRichtungBefehle{hinherRichtung}{$\Leftrightarrow$}

\def\cal#1{\mathcal{#1}}

\def\mathfrak#1{\mbox{\usefont{U}{euf}{m}{n}#1}}

\def\rectangleblack{\text{\RectangleBold}}

\def\squareblack{\blacksquare}

\def\crefname@full#1#2#3#4#5{%
    \crefname{#1}{#2}{#3}
    \Crefname{#1}{#4}{#5}
}
\def\crefname@fullmod#1#2#3#4#5{%
    \crefname@full{#1}{#2}{#3}{#4}{#5}
    \crefname@full{#1@basic}{#2}{#3}{#4}{#5}
    \crefname@full{#1@withName}{#2}{#3}{#4}{#5}
}
\crefname@full{chapter}{chapter}{chapters}{Chapter}{Chapters}
\crefname@full{appendix}{appendix}{appendices}{Appendix}{Appendices}
\crefname@full{section}{section}{sections}{Section}{Sections}
\crefname@full{subsection}{section}{sections}{Section}{Sections}
\crefname@full{subsubsection}{section}{sections}{Section}{Sections}
\crefname@full{subsubsubsection}{section}{sections}{Section}{Sections}
\crefname@full{table}{table}{tables}{Table}{Tables}
\crefname@full{figure}{figure}{figures}{Figure}{Figures}
\crefname@full{subfigure}{figure}{figures}{Figure}{Figures}

\crefname@fullmod{thm}{theorem}{theorems}{Theorem}{Theorems}
\crefname@fullmod{thmStar}{theorem}{theorems}{Theorem}{Theorems}
\crefname@fullmod{conj}{conjecture}{conjectures}{Conjecture}{Conjectures}
\crefname@fullmod{cor}{corollary}{corollaries}{Corollary}{Corollaries}
\crefname@fullmod{defn}{definition}{definitions}{Definition}{Definitions}
\crefname@fullmod{conv}{convention}{conventions}{Convention}{Conventions}
\crefname@fullmod{e.g.}{example}{examples}{Example}{Examples}
\crefname@fullmod{prop}{proposition}{propositions}{Proposition}{Propositions}
\crefname@fullmod{proof}{proof}{proofs}{Proof}{Proofs}
\crefname@fullmod{lemm}{lemma}{lemmata}{Lemma}{Lemmata}
\crefname@fullmod{qstn}{question}{questions}{Question}{Questions}
\crefname@fullmod{rem}{remark}{remarks}{Remark}{Remarks}

    \def\qedEIGEN#1{\@ifnextchar[{\qedEIGEN@c{#1}}{\qedEIGEN@bes{#1}}}
    \def\qedEIGEN@bes#1{%
        \parfillskip=0pt
        \widowpenalty=10000
        \displaywidowpenalty=10000
        \finalhyphendemerits=0
        \leavevmode
        \unskip
        \nobreak
        \hfil
        \penalty50
        \hskip.2em
        \null
        \hfill
        #1
        \par%
    }
    \def\qedEIGEN@c#1[#2]{%
        \parfillskip=0pt
        \widowpenalty=10000
        \displaywidowpenalty=10000
        \finalhyphendemerits=0
        \leavevmode
        \unskip
        \nobreak
        \hfil
        \penalty50
        \hskip.2em
        \null
        \hfill
        {#1~{\small\bfseries\upshape (#2)}}%
        \par%
    }
    \def\qedVARIANT#1#2{
        \expandafter\def\csname ennde#1Sign\endcsname{#2}
        \expandafter\def\csname ennde#1\endcsname{\@ifnextchar[{\qedEIGEN@c{#2}}{\qedEIGEN@bes{#2}}} 
    }
    \qedVARIANT{OfProof}{$\squareblack$}
    \qedVARIANT{OfWork}{\rectangleblack}
    \qedVARIANT{OfSomething}{$\dashv$}
    \qedVARIANT{OnNeutral}{} 

    \def\ra@pretheoremwork{
        \setlength{\theorempreskipamount}{\ownspaceabovethm}
    }
    \def\rathmtransfer#1#2{
        \expandafter\def\csname #2\endcsname{\csname #1\endcsname}
        \expandafter\def\csname end#2\endcsname{\csname end#1\endcsname}
    }

    \def\ranewthm#1#2#3[#4]{
        \theoremstyle{\current@theoremstyle}
        \theoremseparator{\current@theoremseparator}
        \theoremprework{\ra@pretheoremwork}
        \@ifundefined{#1@basic}{\newtheorem{#1@basic}[#4]{#2}}{\renewtheorem{#1@basic}[#4]{#2}}
        \theoremstyle{\current@theoremstyle}
        \theoremseparator{\thmForceSepPt}
        \theoremprework{\ra@pretheoremwork}
        \@ifundefined{#1@withName}{\newtheorem{#1@withName}[#4]{#2}}{\renewtheorem{#1@withName}[#4]{#2}}
        \theoremstyle{nonumberplain}
        \theoremseparator{\thmForceSepPt}
        \theoremprework{\ra@pretheoremwork}
        \@ifundefined{#1@star@basic}{\newtheorem{#1@star@basic}[#4]{#2}}{\renewtheorem{#1@star@basic}[#4]{#2}}
        \theoremstyle{nonumberplain}
        \theoremseparator{\thmForceSepPt}
        \theoremprework{\ra@pretheoremwork}
        \@ifundefined{#1@star@withName}{\newtheorem{#1@star@withName}[#4]{#2}}{\renewtheorem{#1@star@withName}[#4]{#2}}
        \umbauenenv{#1}{#3}[#4]
        \umbauenenv{#1@star}{#3}[#4]
        \rathmtransfer{#1@star}{#1*}
    }

    \def\umbauenenv#1#2[#3]{%
        \expandafter\def\csname #1\endcsname{\relax%
            \@ifnextchar[{\csname #1@\endcsname}{\csname #1@\endcsname[*]}%
        }
        \expandafter\def\csname #1@\endcsname[##1]{\relax%
            \@ifnextchar[{\csname #1@@\endcsname[##1]}{\csname #1@@\endcsname[##1][*]}%
        }
        \expandafter\def\csname #1@@\endcsname[##1][##2]{%
            \ifx*##1%
                \def\enndeOfBlock{\csname end#1@basic\endcsname}
                \csname #1@basic\endcsname%
            \else%
                \def\enndeOfBlock{\csname end#1@withName\endcsname}
                \csname #1@withName\endcsname[##1]%
            \fi%
            \def\makelabel####1{%
                \gdef\beweislabel{####1}%
                \label{\beweislabel}%
            }%
            \ifx*##2%
                \def\enndeSymbol{\qedEIGEN{#2}}
            \else%
                \def\enndeSymbol{\qedEIGEN{#2}[##2]}
            \fi
        }
        \expandafter\gdef\csname end#1\endcsname{\enndeSymbol\enndeOfBlock}
    }

        \def\current@theoremstyle{plain}
        \def\current@theoremseparator{\thmnumberingseppt}
        \theoremstyle{\current@theoremstyle}
        \theoremseparator{\current@theoremseparator}
        \theoremsymbol{}

    \generatenestedthmnumbering{arabic}{section}{X}
    \generatenestedthmnumbering{arabic}{section}{Xe}
    \generatenestedthmnumbering{Roman}{section}{Xsp}

        \theoremheaderfont{\upshape\bfseries}
        \theorembodyfont{\slshape}

    \ranewthm{thm}{Theorem}{\enndeOnNeutralSign}[X]
    \ranewthm{lemm}{Lemma}{\enndeOnNeutralSign}[X]
    \ranewthm{cor}{Corollary}{\enndeOnNeutralSign}[X]
    \ranewthm{prop}{Proposition}{\enndeOnNeutralSign}[X]

        \theorembodyfont{\upshape}

    \ranewthm{defn}{Definition}{\enndeOnNeutralSign}[X]
    \ranewthm{conv}{Convention}{\enndeOnNeutralSign}[X]
    \ranewthm{e.g.}{Example}{\enndeOnNeutralSign}[X]
    \ranewthm{fact}{Fact}{\enndeOnNeutralSign}[X]
    \ranewthm{rem}{Remark}{\enndeOnNeutralSign}[X]
    \ranewthm{qstn}{Question}{\enndeOnNeutralSign}[X]

        \theoremheaderfont{\itshape\bfseries}
        \theorembodyfont{\upshape}

    \ranewthm{proof@tmp}{Proof}{\enndeOfProofSign}[Xdisplaynone]
    \rathmtransfer{proof@tmp*}{proof}

    \def\behauptungbeleg@claim{%
        \iflanguage{british}{Claim}{%
        \iflanguage{english}{Claim}{%
        \iflanguage{ngerman}{Behauptung}{%
        \iflanguage{russian}{Утверждение}{%
        Claim%
        }}}}%
    }
    \def\behauptungbeleg@pf@kurz{%
        \iflanguage{british}{Pf}{%
        \iflanguage{english}{Pf}{%
        \iflanguage{ngerman}{Bew}{%
        \iflanguage{russian}{Доказательство}{%
        Pf%
        }}}}%
    }
    \def\behauptungbeleg{\@ifnextchar\bgroup{\behauptungbeleg@c}{\behauptungbeleg@bes}}
            \def\behauptungbeleg@c#1{\item[{\bfseries \behauptungbeleg@claim\erlaubeplatz #1.}]}
            \def\behauptungbeleg@bes{\item[{\bfseries \behauptungbeleg@claim.}]}
        \def\belegbehauptung{\item[{\bfseries\itshape\behauptungbeleg@pf@kurz.}]}

    \newdateformat{standardshort}{\oldstylenums{\THEYEAR}.\oldstylenums{\THEMONTH}.\oldstylenums{\THEDAY}}
    \newdateformat{standardcompact}{\THEYEAR\twodigit{\THEMONTH}\twodigit{\THEDAY}}
    \newdateformat{standardlong}{\THEYEAR\ \monthname\ \THEDAY}
    \newcolumntype{\RECHTS}[1]{>{\raggedleft}p{#1}}
    \newcolumntype{\LINKS}[1]{>{\raggedright}p{#1}}
    \newcolumntype{m}{>{$}l<{$}}
    \newcolumntype{C}{>{$}c<{$}}
    \newcolumntype{L}{>{$}l<{$}}
    \newcolumntype{R}{>{$}r<{$}}
    \newcolumntype{0}{@{\hspace{0pt}}}
    \newcolumntype{\LINKSRAND}{@{\hspace{\@totalleftmargin}}}
    \newcolumntype{h}{@{\extracolsep{\fill}}}
    \newcolumntype{i}{>{\itshape}}
    \newcolumntype{t}{@{\hspace{\tabcolsep}}}
    \newcolumntype{q}{@{\hspace{1em}}}
    \newcolumntype{n}{@{\hspace{-\tabcolsep}}}
    \newcolumntype{M}[2]{%
        >{\begin{minipage}{#2}\begin{math}}%
        {#1}%
        <{\end{math}\end{minipage}}%
    }
    \newcolumntype{T}[2]{%
        >{\begin{minipage}{#2}}%
        {#1}%
        <{\end{minipage}}%
    }
    \def\punkteumgebung@genbefehl#1#2#3{
        \punkteumgebung@genbefehl@{#1}{#2}{#3}{}{}
        \punkteumgebung@genbefehl@{multi#1}{#2}{#3}{
            \setlength{\columnsep}{10pt}%
            \setlength{\columnseprule}{0pt}%
            \begin{multicols}{\thecolumnanzahl}%
        }{\end{multicols}\nvraum{1}}
    }
    \def\punkteumgebung@genbefehl@#1#2#3#4#5{
        \expandafter\gdef\csname #1\endcsname{
            \@ifnextchar\bgroup{\csname #1@c\endcsname}{\csname #1@bes\endcsname}
        }
            \expandafter\def\csname #1@c\endcsname##1{
                \@ifnextchar[{\csname #1@c@\endcsname{##1}}{\csname #1@c@\endcsname{##1}[\z@]}
            }
            \expandafter\def\csname #1@c@\endcsname##1[##2]{
                \@ifnextchar[{\csname #1@c@@\endcsname{##1}[##2]}{\csname #1@c@@\endcsname{##1}[##2][\z@]}
            }
            \expandafter\def\csname #1@c@@\endcsname##1[##2][##3]{
                \let\alterlinkerRand\gesamtlinkerRand
                \let\alterrechterRand\gesamtrechterRand
                \addtolength{\gesamtlinkerRand}{##2}
                \addtolength{\gesamtrechterRand}{##3}
                \advance\linewidth -##2%
                \advance\linewidth -##3%
                \advance\@totalleftmargin ##2%
                \parshape\@ne \@totalleftmargin\linewidth%
                #4
                \begin{#2}[\upshape ##1]%
                    \setlength{\parskip}{0.5\baselineskip}\relax%
                    \setlength{\topsep}{\z@}\relax%
                    \setlength{\partopsep}{\z@}\relax%
                    \setlength{\parsep}{\parskip}\relax%
                    \setlength{\itemsep}{#3}\relax%
                    \setlength{\listparindent}{\z@}\relax%
                    \setlength{\itemindent}{\z@}\relax%
            }
            \expandafter\def\csname #1@bes\endcsname{
                \@ifnextchar[{\csname #1@bes@\endcsname}{\csname #1@bes@\endcsname[\z@]}
            }
            \expandafter\def\csname #1@bes@\endcsname[##1]{
                \@ifnextchar[{\csname #1@bes@@\endcsname[##1]}{\csname #1@bes@@\endcsname[##1][\z@]}
            }
            \expandafter\def\csname #1@bes@@\endcsname[##1][##2]{
                \let\alterlinkerRand\gesamtlinkerRand
                \let\alterrechterRand\gesamtrechterRand
                \addtolength{\gesamtlinkerRand}{##1}
                \addtolength{\gesamtrechterRand}{##2}
                \advance\linewidth -##1%
                \advance\linewidth -##2%
                \advance\@totalleftmargin ##1%
                \parshape\@ne \@totalleftmargin\linewidth%
                #4
                \begin{#2}%
                    \setlength{\parskip}{0.5\baselineskip}\relax%
                    \setlength{\topsep}{\z@}\relax%
                    \setlength{\partopsep}{\z@}\relax%
                    \setlength{\parsep}{\parskip}\relax%
                    \setlength{\itemsep}{#3}\relax%
                    \setlength{\listparindent}{\z@}\relax%
                    \setlength{\itemindent}{\z@}\relax%
            }
        \expandafter\gdef\csname end#1\endcsname{%
            \end{#2}#5
            \setlength{\gesamtlinkerRand}{\alterlinkerRand}
            \setlength{\gesamtlinkerRand}{\alterrechterRand}
        }
    }

    \def\ritempunkt{{\Large \textbullet}} 
    \setdefaultitem{\ritempunkt}{\ritempunkt}{\ritempunkt}{\ritempunkt}
    \punkteumgebung@genbefehl{itemise}{compactitem}{\parskip}{}{}
    \punkteumgebung@genbefehl{kompaktitem}{compactitem}{\z@}{}{}
    \punkteumgebung@genbefehl{enumerate}{compactenum}{\parskip}{}{}
    \punkteumgebung@genbefehl{kompaktenum}{compactenum}{\z@}{}{}

    \renewenvironment{thebibliography}[1]{%
        \begin{ALTthebibliography}{#1}
        \addcontentsline{toc}{part}{\bibname}
    }{%
        \end{ALTthebibliography}
    }

\def\displaysum_#1{\@ifnextchar^{\displaysum@both_{#1}}{\displaysum@@sub{#1}}}
    \def\displaysum@both_#1^#2{\displaysum@@subsup{#1}{#2}}
    \def\displaysum@@sub#1{\mathop{\displaystyle\csname sum\endcsname_{#1}}}
    \def\displaysum@@subsup#1#2{\mathop{\displaystyle\csname sum\endcsname_{#1}^{#2}}}
\def\displaysup_#1{\@ifnextchar^{\displaysup@both_{#1}}{\displaysup@@sub{#1}}}
    \def\displaysup@both_#1^#2{\displaysup@@subsup{#1}{#2}}
    \def\displaysup@@sub#1{\mathop{\displaystyle\csname sup\endcsname_{#1}}}
    \def\displaysup@@subsup#1#2{\mathop{\displaystyle\csname sup\endcsname_{#1}^{#2}}}
\def\displaymin_#1{\@ifnextchar^{\displaymin@both_{#1}}{\displaymin@@sub{#1}}}
    \def\displaymin@both_#1^#2{\displaymin@@subsup{#1}{#2}}
    \def\displaymin@@sub#1{\mathop{\displaystyle\csname min\endcsname_{#1}}}
    \def\displaymin@@subsup#1#2{\mathop{\displaystyle\csname min\endcsname_{#1}^{#2}}}
\def\displaymax_#1{\@ifnextchar^{\displaymax@both_{#1}}{\displaymax@@sub{#1}}}
    \def\displaymax@both_#1^#2{\displaymax@@subsup{#1}{#2}}
    \def\displaymax@@sub#1{\mathop{\displaystyle\csname max\endcsname_{#1}}}
    \def\displaymax@@subsup#1#2{\mathop{\displaystyle\csname max\endcsname_{#1}^{#2}}}
\def\displaylim_#1{\@ifnextchar^{\displaylim@both_{#1}}{\displaylim@@sub{#1}}}
    \def\displaylim@both_#1^#2{\displaylim@@subsup{#1}{#2}}
    \def\displaylim@@sub#1{\mathop{\displaystyle\csname lim\endcsname_{#1}}}
    \def\displaylim@@subsup#1#2{\mathop{\displaystyle\csname lim\endcsname_{#1}^{#2}}}
\def\displayliminf_#1{\@ifnextchar^{\displayliminf@both_{#1}}{\displayliminf@@sub{#1}}}
    \def\displayliminf@both_#1^#2{\displayliminf@@subsup{#1}{#2}}
    \def\displayliminf@@sub#1{\mathop{\displaystyle\csname liminf\endcsname_{#1}}}
    \def\displayliminf@@subsup#1#2{\mathop{\displaystyle\csname liminf\endcsname_{#1}^{#2}}}
\def\displaylimsup_#1{\@ifnextchar^{\displaylimsup@both_{#1}}{\displaylimsup@@sub{#1}}}
    \def\displaylimsup@both_#1^#2{\displaylimsup@@subsup{#1}{#2}}
    \def\displaylimsup@@sub#1{\mathop{\displaystyle\csname limsup\endcsname_{#1}}}
    \def\displaylimsup@@subsup#1#2{\mathop{\displaystyle\csname limsup\endcsname_{#1}^{#2}}}

    \def\matrix#1{\left(\begin{array}[mc]{#1}}
        \def\endmatrix{\end{array}\right)}
    \def\smatrix{\left(\begin{smallmatrix}}
        \def\endsmatrix{\end{smallmatrix}\right)}

    \def\multiargrekursiverbefehl#1#2#3#4#5#6#7#8{%
        \expandafter\gdef\csname#1\endcsname #2##1#4{\csname #1@anfang\endcsname##1#3\egroup}
        \expandafter\def\csname #1@anfang\endcsname##1#3{#5##1\@ifnextchar\egroup{\csname #1@ende\endcsname}{#7\csname #1@mitte\endcsname}}
        \expandafter\def\csname #1@mitte\endcsname##1#3{#6##1\@ifnextchar\egroup{\csname #1@ende\endcsname}{#7\csname #1@mitte\endcsname}}
        \expandafter\def\csname #1@ende\endcsname##1{#8}
    }
    \multiargrekursiverbefehl{svektor}{[}{;}{]}{\begin{smatrix}}{}{\\}{\\\end{smatrix}}
    \multiargrekursiverbefehl{vektor}{[}{;}{]}{\begin{matrix}{c}}{}{\\}{\\\end{matrix}}
    \multiargrekursiverbefehl{vektorzeile}{}{,}{;}{}{&}{}{}
    \multiargrekursiverbefehl{matlabmatrix}{[}{;}{]}{\begin{smatrix}\vektorzeile}{\vektorzeile}{;\\}{;\end{smatrix}}

    \def\underbracenodisplay#1{%
        \mathop{\vtop{\m@th\ialign{##\crcr
        $\hfil\displaystyle{#1}\hfil$\crcr
        \noalign{\kern3\p@\nointerlineskip}%
        \upbracefill\crcr\noalign{\kern3\p@}}}}\limits%
    }

    \def\mathe[#1]#2{%
        \ifthenelse{\equal{\boolinmdframed}{\boolwahr}}{}{\begin{escapeeinzug}}
        \noindent%
        \let\eqtagset\boolfalsch
        \let\eqtaglabel\boolleer
        \let\eqtagsymb\boolleer
        \let\alteqtag\eqtag
        \def\eqtag{\@ifnextchar[{\eqtag@loc@}{\eqtag@loc@[*]}}%
        \def\eqtag@loc@[##1]{\@ifnextchar\bgroup{\eqtag@loc@@[##1]}{\eqtag@loc@@[##1]{}}}%
        \def\eqtag@loc@@[##1]##2{%
            \gdef\eqtagset{\boolwahr}
            \gdef\eqtaglabel{##1}
            \gdef\eqtagsymb{##2}
        }%
        \def\verticalalign{}%
            \IfBeginWith{#1}{t}{\def\verticalalign{t}}{}%
            \IfBeginWith{#1}{m}{\def\verticalalign{c}}{}%
            \IfBeginWith{#1}{b}{\def\verticalalign{b}}{}%
        \def\horizontalalign{\null\hfill\null}%
            \IfEndWith{#1}{l}{}{\null\hfill\null}%
            \IfEndWith{#1}{r}{\def\horizontalalign{}}{}%
        \begin{math}
        \begin{array}[\verticalalign]{0#2}%
    }
        \def\endmathe{%
            \end{array}
            \end{math}\horizontalalign%
            \let\eqtag\alteqtag
            \ifthenelse{\equal{\eqtagset}{\boolwahr}}{\eqtag[\eqtaglabel]{\eqtagsymb}}{}
            \ifthenelse{\equal{\boolinmdframed}{\boolwahr}}{}{\end{escapeeinzug}}%
        }

    \def\longmathe[#1]#2{\relax
        \let\altarraystretch\arraystretch
        \renewcommand\arraystretch{1.2}\relax
        \begin{longtable}[#1]{\LINKSRAND #2}
    }
        \def\endlongmathe{
            \end{longtable}
            \renewcommand\arraystretch{\altarraystretch}
        }

    \def\einzug{\@ifnextchar[{\indents@}{\indents@[\z@]}}
        \def\indents@[#1]{\@ifnextchar[{\indents@@[#1]}{\indents@@[#1][\z@]}}
        \def\indents@@[#1][#2]{%
            \begin{list}{}{\relax
                \setlength{\topsep}{\z@}\relax
                \setlength{\partopsep}{\z@}\relax
                \setlength{\parsep}{\parskip}\relax
                \setlength{\listparindent}{\z@}\relax
                \setlength{\itemindent}{\z@}\relax
                \setlength{\leftmargin}{#1}\relax
                \setlength{\rightmargin}{#2}\relax
                \let\alterlinkerRand\gesamtlinkerRand
                \let\alterrechterRand\gesamtrechterRand
                \addtolength{\gesamtlinkerRand}{#1}
                \addtolength{\gesamtrechterRand}{#2}
            }\relax
                \item[]\relax
        }
            \def\endeinzug{%
                \setlength{\gesamtlinkerRand}{\alterlinkerRand}
                \setlength{\gesamtlinkerRand}{\alterrechterRand}
                \end{list}%
            }

    \def\escapeeinzug{\begin{einzug}[-\gesamtlinkerRand][-\gesamtrechterRand]}
        \def\endescapeeinzug{\end{einzug}}

    \def\programmiercode{
        \modulolinenumbers[1]
        \begin{einzug}[\rtab][\rtab]%
        \begin{linenumbers}%
            \fontfamily{cmtt}\fontseries{m}\fontshape{u}\selectfont%
            \setlength{\parskip}{1\baselineskip}%
            \setlength{\parindent}{0pt}%
    }
        \def\endprogrammiercode{
            \end{linenumbers}
            \end{einzug}
        }

    \def\schattiertebox@genbefehl#1#2#3{
        \expandafter\gdef\csname #1\endcsname{%
            \@ifnextchar[{\csname #1@args\endcsname}{\csname #1@args\endcsname[#3]}
        }
            \expandafter\def\csname #1@args\endcsname[##1]{%
                \@ifnextchar[{\csname #1@args@l\endcsname[##1]}{\csname #1@args@n\endcsname[##1]}
            }
            \expandafter\def\csname #1@args@l\endcsname[##1][##2]{%
                \@ifnextchar[{\csname #1@args@l@r\endcsname[##1][##2]}{\csname #1@args@l@n\endcsname[##1][##2]}
            }
            \expandafter\def\csname #1@args@n\endcsname[##1]{%
                \let\boolinmdframed\boolwahr
                \begin{mdframed}[#2leftmargin=0,rightmargin=0,outermargin=0,innermargin=0,##1]
            }
            \expandafter\def\csname #1@args@l@n\endcsname[##1][##2]{%
                \let\boolinmdframed\boolwahr
                \begin{mdframed}[#2leftmargin=##2/2,rightmargin=##2/2,outermargin=##2/2,innermargin=##2/2,##1]
            }
            \expandafter\def\csname #1@args@l@r\endcsname[##1][##2][##3]{%
                \let\boolinmdframed\boolwahr
                \begin{mdframed}[#2leftmargin=##2,rightmargin=##3,outermargin=##2,innermargin=##3,##1]
            }
        \expandafter\gdef\csname end#1\endcsname{%
            \end{mdframed}
            \let\boolinmdframed\boolfalsch
        }
    }
        \schattiertebox@genbefehl{schattiertebox}{
            splittopskip=0,%
            splitbottomskip=0,%
            frametitleaboveskip=0,%
            frametitlebelowskip=0,%
            skipabove=1\baselineskip,%
            skipbelow=1\baselineskip,%
            linewidth=2pt,%
            linecolor=black,%
            roundcorner=4pt,%
        }{
            backgroundcolor=leer,%
            nobreak=true,%
        }

        \schattiertebox@genbefehl{schattierteboxdunn}{
            splittopskip=0,%
            splitbottomskip=0,%
            frametitleaboveskip=0,%
            frametitlebelowskip=0,%
            skipabove=1\baselineskip,%
            skipbelow=1\baselineskip,%
            linewidth=1pt,%
            linecolor=black,%
            roundcorner=2pt,%
        }{
            backgroundcolor=leer,%
            nobreak=true,%
        }

    \def\tikzsetzepfeil#1{%
        \begin{tikzpicture}[remember picture,overlay,>=latex]%
            \draw #1;%
        \end{tikzpicture}%
    }
    
    \def\tikzsetzekreise[#1]#2#3{%
        \tikzsetzepfeil{%
        [rounded corners,#1]%
            ([shift={(-\tabcolsep,0.75\baselineskip)}]#2)%
            rectangle%
            ([shift={(\tabcolsep,-0.5\baselineskip)}]#3)
        }%
    }

    \tikzset{
        >=stealth,
        auto,
        node distance=1cm,
        thick,
        main node/.style={
            circle,draw,font=\sffamily\Large\bfseries,minimum size=0pt
        },
        state/.style={minimum size=0pt}
        loop above right/.style={loop,out=30,in=60,distance=0.5cm},
        loop above left/.style={above left,out=150,in=120,loop},
        loop below right/.style={below right,out=330,in=300,loop},
        loop below left/.style={below left,out=240,in=210,loop},
        itria/.style={
            draw,dashed,shape border uses incircle,
            isosceles triangle,shape border rotate=90,yshift=-1.45cm
        },
        rtria/.style={
            draw,dashed,shape border uses incircle,
            isosceles triangle,isosceles triangle apex angle=90,
            shape border rotate=-45,yshift=0.2cm,xshift=0.5cm
        },
        ritria/.style={
            draw,dashed,shape border uses incircle,
            isosceles triangle,isosceles triangle apex angle=110,
            shape border rotate=-55,yshift=0.1cm
        },
        litria/.style={
            draw,dashed,shape border uses incircle,
            isosceles triangle,isosceles triangle apex angle=110,
            shape border rotate=235,yshift=0.1cm
        }
    }

\lateinabkuerzung{cf}{cf.}
\lateinabkuerzung{Cf}{Cf.}
\lateinabkuerzung{idest}{i.e.}
\lateinabkuerzung{Idest}{I.e.}
\lateinabkuerzung{exempli}{e.g.}
\lateinabkuerzung{Exempli}{E.g.}
\lateinabkuerzung{etcetera}{etc.}
\lateinabkuerzung{etAlia}{et al.}
\lateinabkuerzung{viz}{viz.}
\deutscheabkuerzung{usw}{usw.}
\deutscheabkuerzung{oBdA}{o.\,B.\,d.\,A.}
\deutscheabkuerzung{OBdA}{O.\,B.\,d.\,A.}
\deutscheabkuerzung{oae}{o.\,\"A.}
\deutscheabkuerzung{oE}{o.\,E.}
\deutscheabkuerzung{OE}{O.\,E.}
\deutscheabkuerzung{og}{o.\,g.}
\deutscheabkuerzung{obengen}{o.\,g.}
\deutscheabkuerzung{obenst}{o.\,s.}
\deutscheabkuerzung{untenst}{u.\,s.}
\deutscheabkuerzung{Tfae}{T.\,f.\,a.\,e.}
\deutscheabkuerzung{tfae}{t.\,f.\,a.\,e.}
\deutscheabkuerzung{wrt}{wrt.}
\deutscheabkuerzung{withoutlog}{w.\,l.\,o.\,g.} 
\deutscheabkuerzung{Withoutlog}{W.\,l.\,o.\,g.} 
\deutscheabkuerzung{respectively}{resp.}
\newcommand{\usesinglequotes}[1]{`#1'}

\def\Proj{\mathop{\mathrm{Proj}}}
\def\topInterior#1{\mathop{\textup{int}}(#1)}
\def\oBall#1_#2{\cal{B}_{#2}(#1)}
\def\clBall#1_#2{\quer{\cal{B}}_{#2}(#1)}

\def\Cts{\@ifnextchar_{\Cts@tief}{\Cts@tief_{}}}
    \def\Cts@tief_#1#2{\@ifnextchar\bgroup{\Cts@two_{#1}{#2}}{\Cts@one_{#1}{#2}}}
    \def\Cts@one_#1#2{C_{#1}\big(#2\big)}
    \def\Cts@two_#1#2#3{C_{#1}\big(#2,~#3\big)}
    
\def\Hom#1#2{\mathrm{\textup{Hom}}\big(#1,~#2\big)}

\def\KmpRm#1{\cal{K}(#1)}
\def\id{\mathop{\textit{id}}}

\def\reals{\mathbb{R}}
\def\kmplx{\mathbb{C}}

\def\uDisc{\cal{D}_{1}}
\def\rtnl{\mathbb{Q}}
\def\intgr{\mathbb{Z}}
\def\onematrix{\text{\upshape\bfseries I}}
\def\zeromatrix{\text{\upshape\bfseries 0}}
\def\zerovector{\text{\upshape\bfseries 0}}
\def\symmdiff{\mathbin{\Delta}}

\def\UpperTrUnit{\mathop{\mathrm{UT}_{1}}}
\def\Heisenberg#1{\mathop{\mathrm{H}_{#1}}}

\def\ExampleGroupHeisenberg{G_{h}}
\def\ExampleMonoidHeisenberg{M_{h}}
\def\ExampleGroupUTunit{G_{u}}
\def\ExampleMonoidUT{M_{u}}

\def\ntrlpos{\mathbb{N}}
\def\ntrlzero{\mathbb{N}_{0}}
\def\realsNonNeg{\reals_{+}}

\def\ntrlpos{\mathbb{N}}
\def\leer{\emptyset}

\def\BRAKET#1#2{\langle{}#1,~#2{}\rangle}

\def\dee{\textup{d}}
\def\einser{\text{\textbf{1}}}
\def\C0{\ensuremath{C_{0}}}
\def\restr#1{\vert_{#1}}
\def\ohne{\setminus}

\def\eps{\varepsilon}
\let\altphi\phi
\let\altvarphi\varphi
    \def\phi{\altvarphi}
    \def\varphi{\altphi}

\def\quer#1{\overline{#1}}

\def\linspann{\mathop{\textup{lin}}}

\def\lim{\mathop{\ell\mathrm{im}}}

\def\dim{\mathop{\textup{dim}}}

\def\ran{\mathop{\textup{ran}}}

\def\co{\mathop{\textup{co}}}

\def\BoundedOps#1{\@ifnextchar\bgroup{\BoundedOps@two{#1}}{\mathop{\mathfrak{L}}(#1)}}
    \def\BoundedOps@two#1#2{\mathop{\mathfrak{L}}(#1,#2)}
\def\HilbertRaum{\mathcal{H}}
\def\BanachRaum{\mathcal{E}}
\def\RaumX{X}
\def\RaumY{Y}

\def\compactcover{\tilde{\cal{K}}}
\def\hardSigma02{\cal{Q}}

    \def\OpSpaceU#1{\mathop{\cal{U}}(#1)}
    \def\OpSpaceI#1{\mathop{\cal{I}}(#1)}
    \def\OpSpaceC#1{\mathop{\cal{C}}(#1)}

\def\SpCs{\cal{F}^{c}_{s}}
\def\SpCw{\cal{F}^{c}_{w}}

\def\SpHs{\cal{C}_{s}}
\def\SpHw{\cal{C}_{w}}

    \def\topWOT{\text{\upshape \scshape wot}}

    \def\toplocWOT{\text{{{$\mathpzc{k}$}}}_{\text{\tiny\upshape \scshape wot}}}
    
    \def\tinytoplocWOT{\text{\scriptsize{{{$\mathpzc{k}$}}}-{\upshape \scshape wot}}}
    \def\topSOT{\text{\upshape \scshape sot}}

    \def\toplocSOT{\text{{{$\mathpzc{k}$}}}_{\text{\tiny\upshape \scshape sot}}}
    
    \def\tinytoplocSOT{\text{\scriptsize{{{$\mathpzc{k}$}}}-{\upshape \scshape sot}}}

\@ifundefined{setcitestyle}{%
}{%
    \setcitestyle{numeric-comp,open={[},close={]}}
}

\renewcommand{\arraystretch}{1}
\def\firstparagraph{\noindent}
\def\continueparagraph{\noindent}

\hypersetup{
    hidelinks=true,
}

    \generatenestedsecnumbering{arabic}{section}{subsection}
    \generatenestedsecnumbering{arabic}{subsection}{subsubsection}
    \def\theunitnamesection{\thesection}

    \def\sectionname{}

    \let\appendix@orig\appendix
    \def\appendix{%
        \appendix@orig%
        \let\boolinappendix\boolwahr
        \addcontentsline{toc}{part}{\appendixname}%
        \addtocontents{toc}{\protect\setcounter{tocdepth}{0}}
        \def\sectionname{Appendix}%
        \def\theunitnamesection{\Alph{section}}%
    }
    \def\notappendix{%
        \let\boolinappendix\boolfalse
        \addtocontents{toc}{\protect\setcounter{tocdepth}{1 }}
        \def\sectionname{}%
        \def\theunitnamesection{\arabic{section}}%
    }

    \def\@seccntformat#1{%
        \protect\textup{%
            \protect\@secnumfont
            \expandafter\protect\csname format#1\endcsname%
            \csname the#1\endcsname
            \expandafter\protect\csname format#1@pt\endcsname%
            \space
        }%
    }

    \def\formatsection@text{\centering\Large\scshape}
    \def\formatsection@pt{\secnumberingseppt}
    \def\section{\@startsection{section}{1}{\z@}{.7\linespacing\@plus\linespacing}{.5\linespacing}{\formatsection@text}}

    \def\formatsubsection@text{\flushleft\bfseries\scshape}
    \def\formatsubsection@pt{\subsecnumberingseppt}
    \def\subsection{\@startsection{subsection}{2}{\z@}{\z@}{\z@\hspace{1em}}{\formatsubsection@text}}

\def\rafootnotectr{20}
\def\incrftnotectr#1{%
    \addtocounter{#1}{1}%
    \ifnum\value{#1}>\rafootnotectr\relax
        \setcounter{#1}{0}%
    \fi%
}
\def\footnoteref[#1]{\protected@xdef\@thefnmark{\ref{#1}}\@footnotemark}
\let\altfootnotetext\footnotetext
    \def\footnotetext[#1]#2{\incrftnotectr{footnote}\altfootnotetext[\value{footnote}]{\label{#1}#2}}

    \def\footnotemark[#1]{\text{\textsuperscript{\getrefnumber{#1}}}}

\DefineFNsymbols*{custom}{abcdefghijklmnopqrstuvwxyz}
\setfnsymbol{custom}

\def\kopfzeiledefault{
    \lhead[]{}
    \lhead[]{}
    \chead[]{}
    \rhead[]{}
    \lfoot[]{}
    \cfoot{\footnotesize\thepage}
    \rfoot[]{}
}

\def\aktuellesfont{\csnamermfamily\endcsname}
\def\documentfont{%
    \gdef\aktuellesfont{\csnamermfamily\endcsname}%
    \fontfamily{cmr}\fontseries{m}\selectfont%
    \renewcommand{\sfdefault}{phv}%
    \renewcommand{\ttdefault}{pcr}%
    \renewcommand{\rmdefault}{cmr}
    \renewcommand{\bfdefault}{bx}%
    \renewcommand{\itdefault}{it}%
    \renewcommand{\sldefault}{sl}%
    \renewcommand{\scdefault}{sc}%
    \renewcommand{\updefault}{n}%
}

\allowdisplaybreaks

\def\startdocumentlayoutoptions{
    \selectlanguage{british}
    \setlength{\parskip}{0.25\baselineskip}
    \setlength{\parindent}{2em}
    \kopfzeiledefault
    \documentfont
    \normalsize
}

\def\highlightTerm#1{\emph{#1}}
\newcommand{\highlightForReview}[1]{%
    \bgroup\color{blue}#1\egroup%
}

\def\@adminfootnotes{%
    \let\@makefnmark\relax
    \let\@thefnmark\relax
    \ifx\@empty\@date\else%
        \@footnotetext{\@setdate}%
    \fi%
    \ifx\@empty\@subjclass\else%
        \@footnotetext{\@setsubjclass}%
    \fi
    \ifx\@empty\@keywords\else%
        \@footnotetext{\@setkeywords}%
    \fi
    \ifx\@empty\thankses\else%
        \@footnotetext{\def\par{\let\par\@par}\@setthanks}%
    \fi
}

\def\@settitle{\Large\bfseries\scshape\@title}

\def\@maketitle{%
  \normalfont\normalsize
  \@adminfootnotes
  \@mkboth{\@nx\shortauthors}{\@nx\shorttitle}%
  \global\topskip42\p@\relax
  {\centering\@settitle}
  \ifx\@empty\authors\else{\centering\small\@setauthors}\fi
  \ifx\@empty\@date\else{\vtop{\centering\small\@date\@@par}}\fi
  \ifx\@empty\@dedicatory%
  \else%
    \baselineskip\p@
    \vtop{\centering{\footnotesize\itshape\@dedicatory\@@par}%
    \global\dimen@i\prevdepth}\prevdepth\dimen@i%
  \fi
  \@setabstract
  \normalsize
  \if@titlepage
    \newpage
  \else
    \dimen@34\p@\advance\dimen@-\baselineskip
  \fi
}

\def\addresseshere{%
  \bgroup
  \setlength{\parindent}{0pt}
  \enddoc@text
  \egroup
  \let\enddoc@text\relax
}

\makeatother

\begin{document}
\startdocumentlayoutoptions

\thispagestyle{plain}



\def\abstractname{Abstract}
\begin{abstract}
    The space of unitary $\C0$-semigroups on separable infinite-dimensional Hilbert space,
    when viewed under the topology of uniform weak operator convergence on compact subsets of $\realsNonNeg$,
    is known to admit various interesting residual subspaces.
    Before treating the contractive case, the problem of
    the complete metrisability of this space was raised in
        \cite{eisner2010buchStableOpAndSemigroups}.
    Utilising Borel complexity computations and automatic continuity results for semigroups,
    we obtain a general result, which in particular implies that the
    one-/multiparameter contractive $\C0$-semigroups constitute Polish spaces
    and thus positively addresses the open problem.
\end{abstract}



\title[On the complete metrisability of spaces of contractive semigroups]{On the complete metrisability of spaces of contractive semigroups}
\author{Raj Dahya}
\email{raj.dahya@web.de}
\address{Fakult\"at f\"ur Mathematik und Informatik\newline
Universit\"at Leipzig, Augustusplatz 10, D-04109 Leipzig, Germany}

\maketitle



\setcounternach{section}{1}



\section[Introduction]{Introduction}
\label{sec:intro}

\firstparagraph
In
    \cite{dahya2022weakproblem}
the space of contractive $\C0$-semigroups over a separable infinite-dimensional Hilbert space,
and when viewed with the topology of uniform weak operator convergence on compact subsets of $\realsNonNeg$,
was shown to constitute a Baire space.
The main application of this is
    \cite[Proposition 5.1]{dahya2022weakproblem},
    which relies on the approximation result in
    \cite[Theorem~2.1]{krol2009}
and shows that residual properties for the unitary case automatically transfer to the contractive case.
In particular, this application renders meaningful the residuality results achieved in
    \cite{eisnersereny2009catThmStableSemigroups},
    \cite[Corollary~3.2]{krol2009},
    and \cite[\S{}III.6 and \S{}IV.3.3]{eisner2010buchStableOpAndSemigroups}.
Note also that residual properties of (contractive) operators on Banach spaces,
as initiated in \cite{eisner2010typicalContraction,eisnermaitrai2010typicalOperators},
have recently been studied in connection with
\emph{hypercyclicity} and
the \emph{Invariant Subspace Problem}
in \cite{grivaux2021localspecLp,grivaux2021invariantsubspaceLp,grivaux2021typicalexamples}.
The continuous case remains to be investigated.

In this paper we improve upon the result in \cite{dahya2022weakproblem}
and show that the space of contractive $\C0$-semigroups is \highlightTerm{Polish} (\idest separable, completely metrisable).
In fact, we prove this for spaces of more generally defined semigroups, including multiparameter semigroups.
In particular, our result positively solves a problem raised in
    \cite[\S{}III.6.3]{eisner2010buchStableOpAndSemigroups}
    (\cf also \cite[Remark~2.2]{Eisner2008kato}).
There it was shown that,
when viewed under the topology of uniform weak operator convergence on compact time intervals,
the space of contractive $\C0$-semigroups
is not sequentially closed within the larger space of continuous contraction-valued functions.
We shall reinforce this by studying the geometric properties of these spaces in a general setting,
and providing a deeper reason for this failure (see \Cref{cor:broad-class-of-counterexamples-to-Hs-closed:sig:article-str-problem-raj-dahya}).
This renders the complete metrisability problem non-trivial.

The approach in
    \cite[Theorem~1.20]{dahya2022weakproblem}
involves studying and transferring properties from the subspace of unitary semigroups,
which is a Baire space.
This method crucially relies on the fact that contractive semigroups
can be weakly approximated by unitary semigroups.
These density results in turn arise from the theory of dilations
(\cf
    \cite[Theorem~1]{peller1981estimatesOperatorPolyLp}
    and
    \cite[Theorem~2.1]{krol2009}%
).
By contrast, the approach here bypasses dependency upon dilation.
Instead we directly classify
the space of contractive $\C0$-semigroups
in terms of its Borel complexity within a larger, completely metrisable space.
This complexity result in turn implies complete metrisability
(see \Cref{thm:main-result:sig:article-str-problem-raj-dahya}).

Our result encompasses a broad class of spaces on which the semigroups are defined.
We provide basic examples in the main text
and broaden this to a larger class in \Cref{app:continuity}.
The generality of the main result may also be of interest to other fields.
Multiparameter semigroups, for example, occur in
    structure theorems
        (see \exempli \cite{lopushanskyMultiparamFourier}),
    the study of diffusion equations in space-time dynamics
        (see \exempli \cite{zelik2004}),
    the approximation of periodic functions in multiple variables
        (see \exempli \cite{terehin1975}),
\etcetera.



\section[Definitions of spaces of semigroups]{Definitions of spaces of semigroups}
\label{sec:intro-definitions}

\firstparagraph
Throughout this paper, $\HilbertRaum$ shall denote a fixed separable infinite-dimensional Hilbert space.
Furthermore,

    \begin{mathe}[mc]{rcccccl}
        \BoundedOps{\HilbertRaum}
            &\supseteq &\OpSpaceC{\HilbertRaum}
            &\supseteq &\OpSpaceI{\HilbertRaum}
            &\supseteq &\OpSpaceU{\HilbertRaum}\\
    \end{mathe}

\continueparagraph
denote (from left to right) the spaces of bounded linear operators, contractions, isometries, and unitaries over $\HilbertRaum$.
These can be endowed with the weak operator topology ($\topWOT$)
or the strong operator topology ($\topSOT$).

Instead of working with semigroups defined on $\realsNonNeg$ (continuous time)
or over $\ntrlzero$ (discrete time), we shall more generally work with semigroups parameterised by a topological monoid.

\begin{defn}
    \makelabel{defn:semigroup-on-h:sig:article-str-problem-raj-dahya}
    Let $(M,\cdot,1)$ be a topological monoid.
    A \highlightTerm{semigroup} over $\HilbertRaum$ on $M$
    shall mean any operator-valued function, ${T:M\to\BoundedOps{\HilbertRaum}}$,
    satisfying ${T(1)=\onematrix}$ and ${T(st)=T(s)T(t)}$ for ${s,t\in M}$.
\end{defn}

In other words, semigroups are just certain kinds of algebraic morphisms.
Observe that the above definition applied to the topological monoid
    $(\realsNonNeg,+,0)$
yields the usual definition of an operator semigroup.

The continuous contractive semigroups defined on $M$ may be viewed as subspaces
of the function spaces $\Cts{M}{\OpSpaceC{\HilbertRaum}}$,
where $\OpSpaceC{\HilbertRaum}$
may be endowed with either the $\topWOT$- or $\topSOT$-topology.
We summarise these spaces and their topologies as follows:

\begin{defn}
    \makelabel{defn:standard-funct-spaces:sig:article-str-problem-raj-dahya}
    Let $M$ be a topological monoid.
    Denote via
        ${\SpCs(M)\colonequals\Cts{M}{(\OpSpaceC{\HilbertRaum},\topSOT)}}$
        the $\topSOT$-continuous contraction-valued functions
        defined on $M$,
    and via ${\SpHs(M)\colonequals\Hom{M}{(\OpSpaceC{\HilbertRaum},\topSOT)}}$
        the $\topSOT$-continuous contractive semigroups
            over $\HilbertRaum$ on $M$.
    Denote via $\SpCw(M)$ and $\SpHw(M)$ the respective $\topWOT$-continuous counterparts.
\end{defn}

\begin{defn}
\makelabel{defn:loc-wot-sot-convergence:sig:article-str-problem-raj-dahya}
    Let $M$ be a topological monoid
    and let $\RaumX$ be any of the spaces in \Cref{defn:standard-funct-spaces:sig:article-str-problem-raj-dahya}.
    Let $\KmpRm{M}$ denote the collection of compact subsets of $M$.
    The topologies of
        \highlightTerm{uniform \topWOT-convergence on compact subsets of $M$}
        (short: $\toplocWOT$),
    \respectively of
        \highlightTerm{uniform \topSOT-convergence on compact subsets of $M$}
        (short: $\toplocSOT$)
    are induced via the convergence conditions:

    \begin{mathe}[mc]{rcll}
        T_{i} \overset{\tinytoplocWOT}{\longrightarrow} T
            &:\Longleftrightarrow
                &\forall{\xi,\eta\in\HilbertRaum:~}
                \forall{K\in\KmpRm{M}:~}
                &\sup_{t\in K}|\BRAKET{(T_{i}(t)-T(t))\xi}{\eta}|\underset{i}{\longrightarrow}0\\
        T_{i} \overset{\tinytoplocSOT}{\longrightarrow} T
            &:\Longleftrightarrow
                &\forall{\xi\in\HilbertRaum:~}
                \forall{K\in\KmpRm{M}:~}
                &\sup_{t\in K}\|(T_{i}(t)-T(t))\xi\|\underset{i}{\longrightarrow}0\\
    \end{mathe}

    \continueparagraph
    for all nets, $(T_{i})_{i}\subseteq\RaumX$ and all $T\in\RaumX$.
\end{defn}

Working with these definitions, one can readily classify some of these spaces as follows:

\begin{prop}
\makelabel{prop:basic:SpC:basic-polish:sig:article-str-problem-raj-dahya}
    Let $M$ be a locally compact Polish monoid.
    Then $(\SpCw(M),\toplocWOT)$,
        $(\SpCs(M),\toplocSOT)$,
        and
        $(\SpHs(M),\toplocSOT)$
    are Polish spaces.
\end{prop}

For a full proof see
    \cite[Propositions~1.16~and~1.18]{dahya2022weakproblem}.
For the reader's convenience, we sketch the arguments here.

    \begin{proof}[of \Cref{prop:basic:SpC:basic-polish:sig:article-str-problem-raj-dahya}]
        First note that the spaces,
            ${(\OpSpaceC{\HilbertRaum},\topWOT)}$
            and
            ${(\OpSpaceC{\HilbertRaum},\topSOT)}$,
        are well-known to be Polish
        (%
            see \exempli
            \cite[Exercise~3.4~(5) and Exercise~4.9]{kech1994}%
        ).
        To prove the first claim, we need to show that
            $(\Cts{M}{\RaumY},\toplocWOT)$
        is Polish, where ${\RaumY\colonequals(\OpSpaceC{\HilbertRaum},\topWOT)}$.
        Since $M$ is a locally compact Polish space,
        and since for metrisable spaces, separability is equivalent to second countability,
        one can readily construct a countable collection of compact subsets,
            ${\compactcover\subseteq\KmpRm{M}}$,
        such that $\{\topInterior{K}\mid K\in\compactcover\}$ covers $M$.
        Consider the spaces $\Cts{K}{\RaumY}$ for $K\in\compactcover$ and endow these with the topology of uniform convergence,
        which makes them Polish spaces
            (see \exempli
                \cite[Theorem~4.19]{kech1994}
                or
                \cite[Lemma~3.96--7,~3.99]{aliprantis2005}%
            ).
        The map

            \begin{mathe}[mc]{rcccl}
                \Psi &: &\Cts{M}{\RaumY} &\to &\prod_{K\in\compactcover}\Cts{K}{\RaumY}\\
                    &&f &\mapsto &(f\restr{K})_{K\in\compactcover}\\
            \end{mathe}

        \continueparagraph
        is clearly well-defined.
        Since $\{\topInterior{K} \mid K\in\compactcover\}$ covers $M$ and $M$ is locally compact,
        the map is also clearly bicontinuous.
        The covering property also guarantees that every coherent sequence of continuous functions
            $(f_{K})_{K\in\compactcover}\in\prod_{K\in\compactcover}\Cts{K}{\RaumY}$
        corresponds to a unique continuous function,
            ${f\colonequals\bigcup_{K\in\compactcover}f_{K}\in\Cts{M}{\RaumY}}$,
        satisfying $\Psi(f)=(f_{K})_{K\in\compactcover}$.
        Thus $\Psi$ is a homeomorphism between $\Cts{M}{\RaumY}$
        and the subspace of coherent sequences of continuous functions.
        Since the product of Polish spaces is Polish
            (see
                \cite[Corollary~3.39]{aliprantis2005}%
            )
        and the subspace of coherent sequences is clearly closed under the product topology,
        it follows that $(\Cts{M}{\RaumY},\toplocWOT)$ is Polish.

        The second claim is obtained in the same manner,
        by replacing
            $\RaumY$ by $(\OpSpaceC{\HilbertRaum},\topSOT)$
            and $\toplocWOT$ by $\toplocSOT$
        above.
        Finally, it is easy to verify that $\SpHs(M)$ is a closed subspace within $(\SpCs(M),\toplocSOT)$,
        and thus that $(\SpHs(M),\toplocSOT)$ too is Polish.
    \end{proof}

The aim of this paper is to show that $(\SpHs(M),\toplocWOT)$ is completely metrisable
for some topological monoids $M$, in particular for $M=(\realsNonNeg,+,0)$.
We can achieve this for a broad class of monoids, by appealing to the following condition:

\begin{defn}
\makelabel{defn:good-monoids:sig:article-str-problem-raj-dahya}
    Call a topological monoid, $M$, \highlightTerm{\usesinglequotes{good}},
    if the contractive $\topWOT$-continuous semigroups
    over $\HilbertRaum$ on $M$
    are automatically $\topSOT$-continuous,
    \idest if $\SpHw(M)=\SpHs(M)$ holds.
\end{defn}

All discrete monoids (including non-commutative ones) are trivially \usesinglequotes{good}.
By a classical result, $(\realsNonNeg,+,0)$ is \usesinglequotes{good}
(\cf
    \cite[Theorem~5.8]{Engel1999}
    as well as
    \cite[Theorem~9.3.1 and Theorem~10.2.1--3]{hillephillips2008}%
).
Furthermore, it is easy to see that \usesinglequotes{good} monoids are closed under products:

\begin{prop}
\makelabel{prop:good-monoids-closed-under-products:sig:article-str-problem-raj-dahya}
    Let $d\in\ntrlpos$ and $M_{1},M_{2},\ldots,M_{d}$ be \usesinglequotes{good} topological Polish monoids.
    Then $\prod_{i=1}^{d}M_{i}$ is \usesinglequotes{good}.
\end{prop}

    \begin{proof}
        Let ${T:\prod_{i=1}^{d}M_{i}\to\OpSpaceC{\HilbertRaum}}$
        be a $\topWOT$-continuous contractive semigroup.
        We need to show that $T$ is $\topSOT$-continuous.
        For each $k\in\{1,2,\ldots,d\}$
        let
            ${\pi_{k}:\prod_{i=1}^{d}M_{i}\to M_{k}}$
        denote the canonical projection,
        which is a (continuous) monoid homomorphism,
        and
        let
            ${r_{k}:M_{k}\to\prod_{i=1}^{d}M_{i}}$
        denote the canonical (continuous) monoid homomorphism
        defined by $r_{k}(t)=(1,1,\ldots,t,\ldots,1)$
        (the $d$-tuple with $t$ in the $k$-th position and identity elements elsewhere)
        for all $t\in M_{k}$.
        For each $k\in\{1,2,\ldots,d\}$
        observe further that
        ${T_{k}:M_{k}\to\OpSpaceC{\HilbertRaum}}$
        define by ${T_{k} \colonequals T\circ r_{k}}$
        is a $\topWOT$-continuous homomorphism.
        That is, each $T_{k}$ is a $\topWOT$-continuous contractive semigroup over $\HilbertRaum$ on $M_{k}$.
        Since each $M_{k}$ is \usesinglequotes{good}, these are $\topSOT$-continuous.
        Observe now, that

            \begin{mathe}[mc]{rcl}
                T(t)
                    &= &T(\prod_{i=1}^{d}r_{k}(\pi_{k}(t)))\\
                    &= &T(r_{1}(\pi_{1}(t)))\cdot T(r_{2}(\pi_{2}(t)))\cdot\ldots\cdot T(r_{d}(\pi_{d}(t)))\\
                    &= &T_{1}(\pi_{1}(t))\cdot T_{2}(\pi_{2}(t))\cdot\ldots\cdot T_{d}(\pi_{d}(t))\\
            \end{mathe}

        \continueparagraph
        holds for all $t\in \prod_{i=1}^{d}M_{i}$.
        Since the algebraic projections are continuous
        and the $T_{k}$ are $\topSOT$-continuous and contractive,
        and since multiplication of contractions is $\topSOT$-continuous,
        it follows that $T$ is $\topSOT$-continuous.
    \end{proof}

Thus we immediately obtain the following examples of \usesinglequotes{good} monoids:

\begin{cor}
\makelabel{cor:multiparameter-auto-cts:sig:article-str-problem-raj-dahya}
    For each $d\in\ntrlpos$ the monoid $\realsNonNeg^{d}$, viewed under pointwise addition, is \usesinglequotes{good}.
\end{cor}

If one more generally considers monoids which are closed subspaces of locally compact Hausdorff topological groups,
a sufficient topological condition exists,
which guarantees that a monoid is \usesinglequotes{good}
(see
    \Cref{app:continuity:defn:conditition-II:sig:article-str-problem-raj-dahya}
    and
    \Cref{app:continuity:thm:generalised-auto-continuity:sig:article-str-problem-raj-dahya}%
).
By
    \Cref{%
        app:continuity:e.g.:extendible-mon:reals:sig:article-str-problem-raj-dahya,%
        app:continuity:e.g.:extendible-mon:p-adic:sig:article-str-problem-raj-dahya,%
        app:continuity:e.g.:extendible-mon:discrete:sig:article-str-problem-raj-dahya,%
        app:continuity:e.g.:extendible-mon:non-comm-non-discrete:sig:article-str-problem-raj-dahya,%
    }
    and
    \Cref{app:continuity:prop:monoids-with-cond-closed-under-fprod:sig:article-str-problem-raj-dahya},
the class of monoids satisfying this condition
is closed under finite products
and includes all discrete monoids,
the non-negative reals under addition $(\realsNonNeg,+,0)$,
the $p$-adic integers under addition $(\intgr_{p},+,0)$ for all $p\in\mathbb{P}$,
and even non-discrete non-commutative monoids including
naturally definable monoids contained within the Heisenberg group of order $2d-3$
for each $d\geq 2$.



\section[The $\toplocWOT$-closure of the space of contractive semigroups]{The $\toplocWOT$-closure of the space of contractive semigroups}
\label{sec:convexity}

\firstparagraph
The simplest approach to demonstrate the complete metrisability of
    $(\SpHs(M),\toplocWOT)$
would be to show that this be a closed subspace within the function space
    $(\SpCw(M),\toplocWOT)$,
which we already know to be Polish (see \Cref{prop:basic:SpC:basic-polish:sig:article-str-problem-raj-dahya}).
In
    \cite[Example~\S{}III.6.10]{eisner2010buchStableOpAndSemigroups}
    and
    \cite[Example~2.1]{Eisner2008kato}
a construction is provided, which demonstrates that this fails
in particular in the case of one-parameter contractive $\C0$-semigroups.
In this section we reveal that the deeper reason for this failure
is that the closure of $\SpHs(M)$ within $(\SpCw(M),\toplocWOT)$ is always convex,
whereas for a broad class of topological monoids, $M$, the subset $\SpHs(M)$ is not convex
(see \Cref{cor:broad-class-of-counterexamples-to-Hs-closed:sig:article-str-problem-raj-dahya} below).

Before we proceed, we require a few definitions.
In the following $M$ shall denote an arbitrary topological monoid.
We continue to use $\HilbertRaum$ to denote a separable infinite-dimensional Hilbert space
and $\onematrix_{\HilbertRaum}$ (or simply $\onematrix$) for the identity operator.

\begin{defn}
\makelabel{defn:partition-of-identity-operator:sig:article-str-problem-raj-dahya}
    For $n\in\ntrlpos$ and
        ${u_{1},u_{2},\ldots,u_{n}\in\BoundedOps{\HilbertRaum}}$
    we shall call
        $(u_{1},u_{2},\ldots,u_{n})$
    an \highlightTerm{isometric partition of the identity},
    just in case the following axioms hold:

    \begin{kompaktenum}{(\bfseries {P}1)}[\rtab][\rtab]
    \item\label{ax:partition:1}
        $u_{j}^{\ast}u_{i}=\delta_{ij}\cdot\onematrix$ for all $i,j\in\{1,2,\ldots,n\}$.
    \item\label{ax:partition:2}
        $\sum_{i=1}^{n}u_{i}u_{i}^{\ast}=\onematrix$.
    \end{kompaktenum}

    \nvraum{1}

\end{defn}

Note that by axiom
    (P\ref{ax:partition:1})
the operators in an isometric partition are necessarily isometries.

\begin{rem}
\makelabel{rem:isometric-partitions-of-one-correspond-to-decompositions:sig:article-str-problem-raj-dahya}
    Let $n\in\ntrlpos$. If $(u_{1},u_{2},\ldots,u_{n})$ is an isometric partition of $\onematrix$,
    then, letting
        ${\HilbertRaum_{i}\colonequals\ran(u_{i})}$ for $i\in\{1,2,\ldots,n\}$,
    it is easy to see that
        (P\ref{ax:partition:1}) and (P\ref{ax:partition:2}),
    together with the fact that each $u_{i}$ is necessarily an isometry,
    imply that the $\HilbertRaum_{i}$ are mutually orthogonal closed subspaces of $\HilbertRaum$
    and that $\HilbertRaum=\bigoplus_{i=1}^{n}\HilbertRaum_{i}$.
    Conversely, if $\HilbertRaum$ can be decomposed as $\bigoplus_{i=1}^{n}\HilbertRaum_{i}$
    where each $\HilbertRaum_{i}\subseteq\HilbertRaum$ is a closed subspace
    with $\dim(\HilbertRaum_{i})=\dim(\HilbertRaum)$ for each $i$,
    then letting $u_{i}\in\BoundedOps{\HilbertRaum}$
    be any isometries with $\ran(u_{i})=\HilbertRaum_{i}$
    for each $i\in\{1,2,\ldots,n\}$,
    one can readily see that
        (P\ref{ax:partition:1}) and (P\ref{ax:partition:2})
    are satisfied.
    Thus, isometric partitions of $\onematrix$
    can be constructed from orthogonal decompositions into infinite-dimensional closed subspaces of $\HilbertRaum$
    and vice versa.
\end{rem}

Of course, these observations only apply for infinite-dimensional Hilbert spaces.

\begin{defn}
    Let $n\in\ntrlpos$,
    $(u_{1},u_{2},\ldots,u_{n})$ be an isometric partition of $\onematrix$,
    and $T_{1},T_{2},\ldots,T_{n}\in\SpCw(M)$.
    Denote via $(\bigoplus_{i=1}^{n}T_{i})_{\quer{u}}$
    the operator-valued function ${T:M\to\BoundedOps{\HilbertRaum}}$
    given by

        \begin{mathe}[mc]{rcl}
            T(\cdot) &= &\sum_{i=1}^{n}u_{i} T_{i}(\cdot) u_{i}^{\ast}.\\
        \end{mathe}

    \nvraum{1}

\end{defn}

\begin{defn}
    Let $A\subseteq\SpCw(M)$.
    Say that $A$ is \highlightTerm{closed under finite joins}
    just in case for all $n\in\ntrlpos$,
    all isometric partitions
        $\quer{u} \colonequals (u_{1},u_{2},\ldots,u_{n})$ of $\onematrix$,
    and all $T_{1},T_{2},\ldots,T_{n}\in A$,
    it holds that
    ${(\bigoplus_{i=1}^{n}T_{i})_{\quer{u}}\in A}$.
\end{defn}

The property of being closed under finite joins is a key ingredient
in proving the convexity of the closure of subsets
(see \Cref{lemm:closure-is-convex:sig:article-str-problem-raj-dahya} below).
We first provide some basic observations about which subsets are closed under finite joins.

\begin{prop}
\makelabel{prop:basic-some-subspaces-are-closed-under-finite-stream-addition:sig:article-str-problem-raj-dahya}
    Let $A \in \{ \SpCw(M), \SpCs(M), \SpHw(M), \SpHs(M) \}$.
    Then $A$ is closed under finite joins.
\end{prop}

    \begin{proof}
        First consider the case $A=\SpCw(M)$.
        Let $n\in\ntrlpos$,
        $\quer{u} \colonequals (u_{1},u_{2},\ldots,u_{n})$ be an isometric partition of $\onematrix$,
        and $T_{1},T_{2},\ldots,T_{n}\in A$.
        We need to show that ${T \colonequals (\bigoplus_{i=1}^{n}T_{i})_{\quer{u}}}$ is in $A$.
        Applying the properties of the partition yields

            \begin{mathe}[mc]{rcl}
                \|T(t)\xi\|^{2}
                &= &\sum_{i,j=1}^{n}
                        \BRAKET{u_{i}T_{i}(t)u_{i}^{\ast}\xi}{u_{j}T_{j}(t)u_{j}^{\ast}\xi}\\
                &\textoverset{(P\ref{ax:partition:1})}{=}
                    &\sum_{i=1}^{n}
                        \BRAKET{u_{i}T_{i}(t)u_{i}^{\ast}\xi}{u_{i}T_{i}(t)u_{i}^{\ast}\xi}\\
                &\leq
                    &\sum_{i=1}^{n}
                        \|u_{i}T_{i}(t)\|\|u_{i}^{\ast}\xi\|\\
                &\leq
                    &\sum_{i=1}^{n}
                        \|u_{i}^{\ast}\xi\|\\
                &\multispan{2}{\text{since $u_{i}$ is isometric and $T_{i}(t)$ contractive for all $i$}}\\
                &= &\BRAKET{\sum_{i=1}^{n}u_{i}u_{i}^{\ast}\xi}{\xi}\\
                &\textoverset{(P\ref{ax:partition:2})}{=}
                    &\BRAKET{\onematrix \xi}{\xi}
                = \|\xi\|^{2}\\
            \end{mathe}

        \continueparagraph
        for all $\xi\in\HilbertRaum$ and all $t\in M$.
        And since by construction, ${T(\cdot)=\sum_{i=1}^{n}u_{i}T_{i}(\cdot)u_{i}^{\ast}}$,
        it clearly holds that $T$ is $\topWOT$-continuous.
        Thus $T$ is a $\topWOT$-continuous contraction-valued function, \idest $T\in A$.
        Hence $A$ is closed under finite joins.
        The case of $A=\SpCs(M)$ is analogous.

        Next we consider the case $A=\SpHw(M)$.
        Let $n\in\ntrlpos$,
        $\quer{u} \colonequals (u_{1},u_{2},\ldots,u_{n})$ be an isometric partition of $\onematrix$,
        and $T_{1},T_{2},\ldots,T_{n}\in A$.
        We need to show that ${T \colonequals (\bigoplus_{i=1}^{n}T_{i})_{\quer{u}}}$ is in $A$.
        Since $A\subseteq\SpCw(M)$ and $\SpCw(M)$ is closed under finite joins,
        we already know that $T\in\SpCw(M)$, \idest that $T$ is contraction-valued and $\topWOT$-continuous.
        To show that $T\in A$, it remains to show that $T$ is a semigroup.
        Since each of the $T_{i}$ are semigroups, applying the properties of the partition
        yields

            \begin{mathe}[mc]{rcccccccl}
                T(1)
                &= &\sum_{i=1}^{n}u_{i}T_{i}(1)u_{i}^{\ast}
                &= &\sum_{i=1}^{n}u_{i}\cdot\onematrix\cdot u_{i}^{\ast}
                &= &\sum_{i=1}^{n}u_{i}u_{i}^{\ast}
                &\textoverset{(P\ref{ax:partition:2})}{=}
                    &\onematrix\\
            \end{mathe}

        \continueparagraph
        and

            \begin{mathe}[mc]{rcl}
                T(s)T(t)
                &= &(\sum_{i=1}^{n}u_{i}T_{i}(s)u_{i}^{\ast})(\sum_{j=1}^{n}u_{j}T_{j}(t)u_{j}^{\ast})\\
                &= &\sum_{i,j=1}^{n}u_{i}T_{i}(s)u_{i}^{\ast}\cdot u_{j} T_{j}(t)u_{j}^{\ast}\\
                &\textoverset{(P\ref{ax:partition:1})}{=}
                    &\sum_{i=1}^{n}u_{i}T_{i}(s)T_{i}(t)u_{i}^{\ast}\\
                &= &\sum_{i=1}^{n}u_{i}T_{i}(st)u_{i}^{\ast}\\
                &= &T(st)\\
            \end{mathe}

        \continueparagraph
        for all $s,t\in M$.
        Thus $T$ is a $\topWOT$-continuous contractive semigroup, \idest $T\in A$.
        Hence $A$ is closed under finite joins.
        The case of $A=\SpHs(M)$ is analogous.
    \end{proof}

\begin{prop}
\makelabel{prop:closure-is-convex:sig:article-str-problem-raj-dahya}
    Let $A\subseteq \SpCw(M)$ and
    let $\quer{A}$ be the closure of $A$ within $(\SpCw(M),\toplocWOT)$.
    If $A$ is closed under finite joins,
    then so too is $\quer{A}$.
\end{prop}

    \begin{proof}
        Let $n\in\ntrlpos$,
        $\quer{u} \colonequals (u_{1},u_{2},\ldots,u_{n})$ be an isometric partition of $\onematrix$,
        and $T_{1},T_{2},\ldots,T_{n}\in\quer{A}$.
        We need to show that ${T \colonequals (\bigoplus_{j=1}^{n}T_{j})_{\quer{u}}}$ is in $\quer{A}$.
        To see this, we may simply fix a net
            ${((T^{(i)}_{1},T^{(i)}_{2},\ldots,T^{(i)}_{n}))_{i} \subseteq \prod_{j=1}^{n}A}$
        such that ${T^{(i)}_{j}\underset{i}{\overset{\tinytoplocWOT}{\longrightarrow}} T_{j}}$
        for all $j\in\{1,2,\ldots,n\}$.
        Since $A$ is closed under finite joins,
        we have
            ${T^{(i)}\colonequals (\bigoplus_{j=1}^{n}T^{(i)}_{j})_{\quer{u}}\in A}$
        for all $i$. We also clearly have

                \begin{mathe}[mc]{rcccccl}
                    A \ni T^{(i)}
                    &= &\sum_{j=1}^{n}u_{j}T^{(i)}_{j}(\cdot)u_{j}^{\ast}
                    &\underset{i}{\overset{\tinytoplocWOT}{\longrightarrow}}
                        &\sum_{j=1}^{n}u_{j}T_{j}(\cdot)u_{j}^{\ast}
                        &= &T(\cdot).\\
                \end{mathe}

        \continueparagraph
        Hence $T\in\quer{A}$.
    \end{proof}

\begin{lemm}
\makelabel{lemm:closure-is-convex:sig:article-str-problem-raj-dahya}
    Let $A\subseteq \SpCw(M)$ be closed under finite joins.
    Then the closure, $\quer{A}$, of $A$ within $(\SpCw(M),\toplocWOT)$ is convex.
\end{lemm}

    \begin{proof}
        Since $\HilbertRaum$ is a separable Hilbert space,
        it admits a countable orthonormal basis (ONB) ${B\subseteq\HilbertRaum}$,
        which we shall fix.
        It suffices to show for $S,T\in\quer{A}$ and $\alpha,\beta\in[0,1]$ with ${\alpha+\beta=1}$,
        that ${R \colonequals \alpha S + \beta T \in \quer{A}}$.
        To do this, we need to show that $R$ can be approximated within the $\toplocWOT$-topology
        by elements in $\quer{A}$.
        In order to achieve this, it suffices to
        fix arbitrary
            $K\in\KmpRm{M}$,
            $F\subseteq B$ finite,
        and $\eps>0$,
        and show that some $\tilde{R}\in\quer{A}$ exists
        satisfying

            \begin{mathe}[mc]{lql}
            \eqtag[eq:0:\beweislabel]
                \sup_{t\in K}|\BRAKET{(R(t)-\tilde{R}(t))e}{e'}|<\eps,
                &\text{for all $e,e'\in F$.}
            \end{mathe}

        To construct $\tilde{R}$, we first construct an isometric partition, $(w_{0},w_{1})$, of $\onematrix$,
        such that $w_{0}$ fixes the vectors in $F$.
        This can be achieved as follows:
        Since the ONB $B$ is infinite, a partition $\{B_{0},B_{1}\}$ of $B$
        exists satisfying $F\subseteq B_{0}$ and $|B_{0}|=|B_{1}|=|B|=\dim(\HilbertRaum)$.
        There thus exist bijections ${f_{0}:B\to B_{0}}$ and ${f_{1}:B\to B_{1}}$
        and since $B_{0}\supseteq F$, we may assume without loss of generality that $f_{0}\restr{F}=\id_{F}$.
        Using these bijections we obtain (unique) isometries
            ${w_{0},w_{1}\in\BoundedOps{\HilbertRaum}}$
        satisfying ${w_{0}e=f_{0}(e)}$ and ${w_{1}e=f_{1}(e)}$ for all $e\in B$.
        In particular,
            $\ran(w_{0})=\quer{\linspann}(B_{0})$
        and $\ran(w_{1})=\quer{\linspann}(B_{1})$.
        Now since $\{B_{0},B_{1}\}$ partitions $B$,
        we have $\HilbertRaum=\quer{\linspann}(B)=\quer{\linspann}(B_{0})\oplus\quer{\linspann}(B_{1})$.
        As per \Cref{rem:isometric-partitions-of-one-correspond-to-decompositions:sig:article-str-problem-raj-dahya}
        it follows that
            $(w_{0},w_{1})$
        satisfies the axioms of an isometric partition of $\onematrix$.

        Now set

            \begin{mathe}[mc]{rclcrcl}
                u_{0} &\colonequals &\sqrt{\alpha}w_{0} + \sqrt{\beta}w_{1}
                &\text{and}
                &u_{1} &\colonequals &\sqrt{\beta}w_{0} - \sqrt{\alpha}w_{1}.\\
            \end{mathe}

        \continueparagraph
        One can easily derive from the fact that $(w_{0},w_{1})$ is an isometric partition of $\onematrix$,
        that $\quer{u} \colonequals (u_{0},u_{1})$ also satisfies the axioms of an isometric partition of $\onematrix$.
        Moreover, since $w_{0}$ was chosen to fix the vectors in $F$,
        applying the properties of the partition $(w_{0},w_{1})$ yields

            \begin{mathe}[mc]{rcccccccl}
            \eqtag[eq:sqrt-alpha:\beweislabel]
                u_{0}^{\ast}e
                    &= &u_{0}^{\ast}w_{0}e
                    &= &\sqrt{\alpha}w_{0}^{\ast}w_{0}e+\sqrt{\beta}w_{1}^{\ast}w_{0}e
                    &\textoverset{(P\ref{ax:partition:1})}{=}
                        &\sqrt{\alpha}\onematrix e+\sqrt{\beta}\zeromatrix e
                    &= &\sqrt{\alpha}e\\
                u_{1}^{\ast}e
                    &= &u_{1}^{\ast}w_{0}e
                    &= &\sqrt{\beta}w_{0}^{\ast}w_{0}e-\sqrt{\alpha}w_{1}^{\ast}w_{0}e
                    &\textoverset{(P\ref{ax:partition:1})}{=}
                        &\sqrt{\beta}\onematrix e-\sqrt{\alpha}\zeromatrix e
                    &= &\sqrt{\beta}e\\
            \end{mathe}

        \continueparagraph
        for all $e\in F$.

        Finally set $\tilde{R}\colonequals(S\bigoplus T)_{\quer{u}}$.
        Since $A$ is closed under finite joins,
        by \Cref{prop:closure-is-convex:sig:article-str-problem-raj-dahya}
        $\quer{A}$ is also closed under finite joins,
        and hence the constructed operator-valued function, $\tilde{R}$, lies in $\quer{A}$.
        For all $e,e'\in F$ we obtain

            \begin{mathe}[mc]{rcl}
                \BRAKET{(\tilde{R}(t)-R(t))e}{e'}
                &= &\BRAKET{u_{0}S(t)u_{0}^{\ast}e}{e'}
                    + \BRAKET{u_{1}T(t)u_{1}^{\ast}e}{e'}
                    - \BRAKET{R(t)e}{e'}\\
                &= &\BRAKET{S(t)u_{0}^{\ast}e}{u_{0}^{\ast}e'}
                    + \BRAKET{T(t)u_{1}^{\ast}e}{u_{1}^{\ast}e'}
                    - \BRAKET{(\alpha S(t) + \beta T(t))e}{e'}\\
                &\eqcrefoverset{eq:sqrt-alpha:\beweislabel}{=}
                    &\BRAKET{S(t)\sqrt{\alpha}e}{\sqrt{\alpha}e'}
                    + \BRAKET{T(t)\sqrt{\beta}e}{\sqrt{\beta}e'}
                    - \BRAKET{(\alpha S(t) + \beta T(t))e}{e'}\\
                &= &0\\
            \end{mathe}

        \continueparagraph
        for all $t\in M$.
        In particular we have found an $\tilde{R}\in\quer{A}$
        which clearly satisfies \eqcref{eq:0:\beweislabel}.

        This establishes that the convex hull of $\quer{A}$ is contained in $\quer{A}$.
        Thus $\quer{A}$ is convex.
    \end{proof}

By \Cref{prop:basic-some-subspaces-are-closed-under-finite-stream-addition:sig:article-str-problem-raj-dahya}
and \Cref{lemm:closure-is-convex:sig:article-str-problem-raj-dahya} we thus immediately obtain the general result:

\begin{cor}
\makelabel{cor:closure-of-Hs-is-convex:sig:article-str-problem-raj-dahya}
    For all topological monoids, $M$,
    the closure of $\SpHs(M)$ within $(\SpCw(M),\toplocWOT)$ is convex.
\end{cor}

We now provide a large class of topological monoids, $M$, for which $\SpHs(M)$ is not convex.

\begin{defn}
    Denote by $\onematrix(\cdot)$
    the trivial $\topSOT$-continuous semigroup over $\HilbertRaum$ on $M$,
    which is everywhere equal to the identity operator.
    Say that $M$ has \highlightTerm{non-trivial unitary semigroups} over $\HilbertRaum$,
    just in case there exists a unitary $\topSOT$-continuous semigroup, $U$,
    over $\HilbertRaum$ on $M$
    with ${U \neq \onematrix(\cdot)}$.
\end{defn}

\begin{defn}
    Let $(G,\cdot,{}^{-1},1)$ be a topological group.
    Say that $M\subseteq G$ is a \highlightTerm{topological submonoid},
    just in case $M$ is endowed with the subspace topology,
    contains the neutral element $1$, and is closed under $\cdot$.
\end{defn}

In particular, if $G$ is a topological group and $M\subseteq G$ is a topological submonoid,
then $M$ is itself a topological monoid.

\begin{prop}
\makelabel{prop:gelfand-raikov-monoids:sig:article-str-problem-raj-dahya}
    Let $G$ be a locally compact Polish group and suppose that
    $M\subseteq G$ is a topological submonoid with $M\neq\{1\}$.
    Then $M$ has non-trivial unitary semigroups over $\HilbertRaum$.
\end{prop}

    \begin{proof}
        Let $t_{0}\in M\ohne\{1\}$.
        By the Gelfand-Raikov theorem
            (see \exempli \cite[Theorem~6]{yoshi1949}),
        there exists a Hilbert space $\HilbertRaum_{0}$
        and an irreducible $\topSOT$-continuous unitary representation,
            ${U_{0}:G\to\OpSpaceU{\HilbertRaum_{0}}}$,
        satisfying $U_{0}(t_{0})\neq \onematrix_{\HilbertRaum_{0}}$.
        By irreducibility and since $G$ is separable,
        it necessarily holds that $\dim(\HilbertRaum_{0})\leq\aleph_{0}=\dim(\HilbertRaum)$.
        If $\dim(\HilbertRaum_{0})=\dim(\HilbertRaum)$,
        then we may assume without loss of generality that $\HilbertRaum_{0}=\HilbertRaum$,
        and thus that $U_{0}$ is a representation of $G$ over $\HilbertRaum$.
        If $\dim(\HilbertRaum_{0})$ is finite,
        we may assume that $\HilbertRaum_{0}\subset\HilbertRaum$
        and view the orthogonal complement $\HilbertRaum_{1}\colonequals\HilbertRaum_{0}^{\perp}$
        within $\HilbertRaum$.
        Replacing $U_{0}$ by
            ${t\in G\mapsto U_{0}(t)\oplus \onematrix_{\HilbertRaum_{1}}}$
        yields an $\topSOT$-continuous unitary representation of $G$ over $\HilbertRaum$
        which satisfies $U_{0}(t_{0})\neq\onematrix_{\HilbertRaum}$.
        In both cases, restricting $U_{0}$ to $M$
        yields a non-trivial $\topSOT$-continuous unitary semigroup over $\HilbertRaum$ on $M$.
        Hence $M$ has non-trivial unitary semigroups over $\HilbertRaum$.
    \end{proof}

\begin{lemm}
\makelabel{lemm:non-trivial-Hs-is-nonconvex:sig:article-str-problem-raj-dahya}
    Suppose that $M$ has non-trivial unitary semigroups over $\HilbertRaum$.
    Then $\SpHs(M)$ is not a convex subset of $\SpCw(M)$.
\end{lemm}

    \begin{proof}
        Let ${S\colonequals\onematrix(\cdot)}$ be the trivial $\topSOT$-continuous unitary semigroup over $\HilbertRaum$ on $M$.
        And by non-triviality we may fix some $\topSOT$-continuous unitary semigroup ${T\in\SpHs(M)\ohne\{\onematrix(\cdot)\}}$.
        In particular, $T(t_{0})\neq\onematrix$ for some $t_{0}\in M$, which we shall fix.
        Choose any $\alpha,\beta\in(0,1)$ with $\alpha+\beta=1$.
        It suffices to show that $R \colonequals \alpha S + \beta T \notin \SpHs(M)$.
        Suppose \textit{per contra} that $R \in \SpHs(M)$.
        Then by the semigroup law we have

            \begin{mathe}[mc]{rcl}
                \zeromatrix &= &R(st)-R(s)R(t)\\
                    &= &\big(\alpha S(st)+\beta T(st)\big)
                        -\big(\alpha S(s)+\beta T(s)\big)
                        \big(\alpha S(t)+\beta T(t)\big)\\
                    &= &\big(\alpha S(s)S(t)+\beta T(s)T(t)\big)
                        -\big(\alpha S(s)+\beta T(s)\big)
                        \big(\alpha S(t)+\beta T(t)\big)\\
                    &= &\alpha(1-\alpha)\,S(s)S(t)
                        +\beta(1-\beta)\,T(s)T(t)
                        -\alpha\beta\,S(s)T(t)
                        -\beta\alpha\,T(s)S(t)\\
                    &= &\alpha\beta\,S(s)S(t)
                        +\beta\alpha\,T(s)T(t)
                        -\alpha\beta\,S(s)T(t)
                        -\beta\alpha\,T(s)S(t)\\
                    &= &\alpha\beta\,(S(s)-T(s))(S(t)-T(t)).\\
            \end{mathe}

        \continueparagraph
        for all $s,t\in M$.
        Since $\alpha,\beta \neq 0$ setting $s\colonequals t\colonequals t_{0}$ in the above yields

            \begin{mathe}[mc]{rcl}
                (\onematrix - u)^{2} &= &\zeromatrix,\\
            \end{mathe}

        \continueparagraph
        where $u\colonequals T(t_{0})\in\OpSpaceU{\HilbertRaum}$.
        Since $u$ is unitary, a basic application of the spectral mapping theorem
        yields that the spectrum of $u$ is $\{1\}$.
        By the Gelfand theorem
            (see \cite[Theorem~2.1.10]{murphy1990}),
        it follows that $T(t_{0})=u=\onematrix$, which is a contradiction.
    \end{proof}

Applying
    \Cref{cor:closure-of-Hs-is-convex:sig:article-str-problem-raj-dahya},
    \Cref{prop:gelfand-raikov-monoids:sig:article-str-problem-raj-dahya},
and
    \Cref{lemm:non-trivial-Hs-is-nonconvex:sig:article-str-problem-raj-dahya}
yields:

\begin{cor}
\makelabel{cor:broad-class-of-counterexamples-to-Hs-closed:sig:article-str-problem-raj-dahya}
    Let $G$ be a locally compact Polish group and suppose that
    $M\subseteq G$ is a topological submonoid with $M\neq\{1\}$.
    Then the closure of $\SpHs(M)$ within $(\SpCw(M),\toplocWOT)$ is convex,
    whilst $\SpHs(M)$ itself is not convex.
    In particular, $\SpHs(M)$ is not a closed subspace within $(\SpCw(M),\toplocWOT)$.
\end{cor}

Considering $G=(\reals^{d},+,\zerovector)$ with $d\geq 1$
and ${M \colonequals \realsNonNeg^{d} \subseteq G}$,
the conditions of \Cref{cor:broad-class-of-counterexamples-to-Hs-closed:sig:article-str-problem-raj-dahya} are satisfied.
In particular, the subspace of one-/multiparameter $\C0$-semigroups
is not closed within the larger space of $\topWOT$-continuous contraction-valued functions.



\section[Complete metrisability results]{Complete metrisability results}
\label{sec:results}

\firstparagraph
As indicated in the introduction, we shall demonstrate the complete metrisability of $\SpHs(M)$
by directly classifying its Borel complexity within
the larger Polish space, $(\SpCw(M),\toplocWOT)$.
By the previous section we know that in a very general setting,
    $\SpHs(M)$ is not closed within $(\SpCw(M),\toplocWOT)$.
Hence we require weaker conditions,
which determine when subspaces are completely metrisable.
For this we rely on the following classical result from descriptive set theory
(see
    \cite[Theorem~3.11]{kech1994}
for a proof):

\begin{lemm*}[Alexandroff's lemma]
    Let $\RaumX$ be a completely metrisable space.
    Then $A\subseteq\RaumX$ viewed with the relative topology is completely metrisable
    if and only if it is a $G_{\delta}$-subset of $\RaumX$.
\end{lemm*}

We now present the main result.

\begin{schattierteboxdunn}[%
    backgroundcolor=leer,%
    nobreak=true,%
]
\begin{thm}
    \makelabel{thm:main-result:sig:article-str-problem-raj-dahya}
    Let $\HilbertRaum$ denote a separable infinite-dimensional Hilbert space
    and $M$ be a locally compact Polish monoid.
    If $M$ is \usesinglequotes{good}, then
        the space $\SpHs(M)$ of contractive $\C0$-semigroups over $\HilbertRaum$ on $M$
    is Polish under the $\toplocWOT$-topology.
\end{thm}
\end{schattierteboxdunn}

    \begin{proof}
        Since $\{1\}$ is a compact subset of $M$, it is easy to see that

            \begin{mathe}[mc]{c}
                \RaumX \colonequals \{T\in\SpCw(M) \mid T(1)=\onematrix\}\\
            \end{mathe}

        \continueparagraph
        is a closed subset in $(\SpCw(M),\toplocWOT)$
        and thus $(\RaumX,\toplocWOT)$ is Polish
        (\cf \Cref{prop:basic:SpC:basic-polish:sig:article-str-problem-raj-dahya}).
        By Alexandroff's lemma it thus suffices to prove that $\SpHs(M)$ is a $G_{\delta}$-subset of $\RaumX$.

        To proceed, observe that since $M$ is a locally compact Polish space, it is $\sigma$-compact,
        \idest there exists a countable collection of compact subsets,
        ${\compactcover\subseteq\KmpRm{M}}$,
        such that ${\bigcup_{K\in\compactcover}K=M}$.
        \Withoutlog one may assume that $\compactcover$ is closed under finite unions.
        Since $\HilbertRaum$ is a separable Hilbert space,
        it admits a countable ONB ${B\subseteq\HilbertRaum}$.
        For each finite $F\subseteq B$ define

            \begin{mathe}[mc]{c}
                \pi_{F} \colonequals \Proj_{\linspann(F)} \in \BoundedOps{\HilbertRaum},\\
            \end{mathe}

        \continueparagraph
        \idest, the projection onto the closed subspace generated by $F$.
        Using these, we construct

            \begin{mathe}[mc]{c}
                d_{K,F,e,e'}(T) \colonequals {\displaystyle\sup_{s,t\in K}}|\BRAKET{(T(s)\pi_{F}T(t)-T(st))e}{e'}|\\
            \end{mathe}

        \continueparagraph
        for each
            ${K\in\KmpRm{M}}$,
            ${F\subseteq B}$ finite,
            ${e,e'\in\HilbertRaum}$,
            and
            ${T\in\RaumX}$,
        and

            \begin{mathe}[mc]{rcl}
            \eqtag[eq:sets:\beweislabel]
                V_{\eps;K,F}(\tilde{T}) &\colonequals
                    &{\displaystyle\bigcap_{e,e'\in F}}
                    \{
                        T\in\RaumX
                        \mid
                            {\displaystyle\sup_{t\in K}}
                                |\BRAKET{(T(t)-\tilde{T}(t))e}{e'}|
                            < \eps
                    \},\\
                W_{\eps;K,F,e,e'} &\colonequals &\{T\in\RaumX \mid d_{K,F,e,e'}(T)<\eps\}\\
            \end{mathe}

        \continueparagraph
        for each
            ${\eps>0}$,
            ${K\in\KmpRm{M}}$,
            ${F\subseteq B}$ finite,
            ${e,e'\in\HilbertRaum}$,
        and ${\tilde{T}\in\RaumX}$.
        We can now present our strategy for the rest of the proof:
        To show that $\SpHs(M)$ is a $G_{\delta}$-subset of $\RaumX$,
        it suffices to show
            (I) that the $W$-sets defined in \eqcref{eq:sets:\beweislabel} are open
            and
            (II) that

            \begin{mathe}[tc]{rcccl}
            \eqtag[eq:G-delta-expression:\beweislabel]
                \SpHs(M) &\subseteq
                    &\displaystyle\bigcap_{\substack{%
                        \eps\in\rtnl^{+},~K\in\compactcover,\\
                        e,e'\in B,\\
                        F_{0}\subseteq B~\text{finite}
                    }}\;
                    \displaystyle\bigcup_{\substack{%
                        F\subseteq B~\text{finite}\\
                        \text{s.t.}~F\supseteq F_{0}
                    }}\;
                        W_{\eps;K,F,e,e'}
                    &\subseteq
                    &\SpHw(M).\\
            \end{mathe}

        \continueparagraph
        If these two statements hold, then by assumption of $M$ being \usesinglequotes{good},
        (I) + (II) will yield that $\SpHs(M)=\SpHw(M)$ are equal to a $G_{\delta}$-subset of $\RaumX$,
        which will complete the proof.

        \null

        Towards (I), fix arbitrary
            ${\eps>0}$,
            ${K\in\KmpRm{M}}$,
            ${F\subseteq B}$ finite,
            and
            ${e,e'\in B}$,
        and consider an arbitrary element, ${\tilde{T}\in W_{\eps;K,F,e,e'}}$.
        We need to show that $\tilde{T}$ is in the interior of $W_{\eps;K,F,e,e'}$.
        By continuity of multiplication in the topological monoid, $M$, the set
            ${K\cdot K=\{st\mid s,t\in K\}}$
        is compact.
        Setting
            ${K'\colonequals K\cup (K\cdot K)}$ and ${F'\colonequals F\cup\{e,e'\}}$,
        it suffices to show that

            \begin{mathe}[tc]{c}
            \eqtag[eq:openness-of-W:\beweislabel]
                V_{\eps';K',F'}(\tilde{T}) \subseteq W_{\eps;K,F,e,e'}\\
            \end{mathe}

        \continueparagraph
        holds for some ${\eps'>0}$, since clearly $\tilde{T}$ is an element of the left hand side
        and by definition of the $\toplocWOT$-topology, the $V$-sets are clearly open.

        We determine $\eps'$ as follows.
        First note that by virtue of $\tilde{T}$ being in $W_{\eps;K,F,e,e'}$

            \begin{mathe}[tc]{c}
            \eqtag[eq:modified-eps:\beweislabel]
                r \colonequals \eps - d_{K,F,e,e'}(\tilde{T}) > 0\\
            \end{mathe}

        \continueparagraph
        holds.
        Since the unit disc, ${\uDisc=\{z\in\kmplx\mid |z|\leq 1\}}$,
        is compact, the map
            ${(a,b)\in\uDisc^{2}\mapsto ab\in\kmplx}$
        is uniformly continuous, and hence some ${\eps'>0}$ exists,
        such that

            \begin{mathe}[tc]{c}
            \eqtag[eq:uniform-cts-func:\beweislabel]
                |a'b'-ab| < \frac{r}{4|F|+1}\\
            \end{mathe}

        \continueparagraph
        for all ${a,b,a',b'\in\uDisc}$
        with ${|a-a'|<\eps'}$ and ${|b-b'|<\eps'}$.
        We may also assume without loss of generality that $\eps'<\frac{r}{4}$.
        With this $\eps'$-value, the left hand side of \eqcref{eq:openness-of-W:\beweislabel} is now determined.
        It remains to show that the inclusion holds.

        Since the elements in $\RaumX$ are all contraction-valued functions
        and the ONB, $B$, consists of unit vectors,
        it holds that
            ${\BRAKET{T(t)\xi}{\eta}\in\uDisc}$
            for all
                $T\in\RaumX$,
                ${t\in M}$,
                and
                ${\xi,\eta\in B}$.
        Now consider an arbitrary $T$ in the left hand side of \eqcref{eq:openness-of-W:\beweislabel}.
        Let $s,t\in K$ be arbitrary.
        Then $s,t,st\in K'$,
        so that by the choice of $F'$
        and by virtue of $T$ being inside $V_{\eps';K',F'}(\tilde{T})$,
        we have

            \begin{mathe}[mc]{rcl}
                |\BRAKET{T(s)e''}{e'}-\BRAKET{\tilde{T}(s)e''}{e'}| &< &\eps',\\
                |\BRAKET{T(t)e}{e''}-\BRAKET{\tilde{T}(t)e}{e''}| &< &\eps',\\
                |\BRAKET{T(st)e}{e'}-\BRAKET{\tilde{T}(st)e}{e'}| &< &\eps'\\
            \end{mathe}

        \continueparagraph
        for all $e''\in F$.
        Since $F$ is an orthonormal collection,
        the choice of $\eps'$ and \eqcref{eq:uniform-cts-func:\beweislabel}
        yield

            \begin{longmathe}[mc]{RCL}
                \begin{array}[b]{0l}
                    |\BRAKET{(T(s)\pi_{F}T(t)-T(st))e}{e'}\\
                    \;-\BRAKET{(\tilde{T}(s)\pi_{F}\tilde{T}(t)-\tilde{T}(st))e}{e'}|\\
                \end{array}
                &\leq &{\displaystyle\sum_{e''\in F}}
                    |
                        \BRAKET{T(s)e''}{e'}
                        \BRAKET{T(t)e}{e''}
                        -
                        \BRAKET{\tilde{T}(s)e''}{e'}
                        \BRAKET{\tilde{T}(t)e}{e''}
                    |\\
                    &&\quad +\;|\BRAKET{T(st)e}{e'}-\BRAKET{\tilde{T}(st)e}{e'}|\\
                &\eqcrefoverset{eq:uniform-cts-func:\beweislabel}{<}
                    &{\displaystyle\sum_{e''\in F}}\frac{r}{4|F|+1}\;+\;\eps'
                <\frac{r|F|}{4|F|+1} + \frac{r}{4}
                <\frac{r}{2}\\
            \end{longmathe}

        \continueparagraph
        for all $s,t\in K$. Thus

            \begin{longmathe}[mc]{RCL}
                d_{K,F,e,e'}(T)
                &= &{\displaystyle\sup_{s,t\in K}}|\BRAKET{(T(s)\pi_{F}T(t)-T(st))e}{e'}|\\
                &\leq &{\displaystyle\sup_{s,t\in K}}|\BRAKET{(\tilde{T}(s)\pi_{F}\tilde{T}(t)-\tilde{T}(st))e}{e'}|
                    +\frac{r}{2}\\
                &= &d_{K,F,e,e'}(\tilde{T})+\frac{r}{2}\\
                &\eqcrefoverset{eq:modified-eps:\beweislabel}{=}
                    &\eps-r+\frac{r}{2}
                <
                    \eps,\\
            \end{longmathe}

        \continueparagraph
        whence ${T\in W_{\eps;K,F,e,e'}}$.
        Hence the inclusion in \eqcref{eq:openness-of-W:\beweislabel} holds, as desired.

        \null

        To prove (II), consider the first inclusion of \eqcref{eq:G-delta-expression:\beweislabel}.
        Let ${T\in\SpHs(M)}$ be arbitrary.
        To show that $T$ is in the $G_{\delta}$-set in the middle of \eqcref{eq:G-delta-expression:\beweislabel},
        consider arbitrary fixed
            ${\eps>0}$,
            ${K\in\KmpRm{M}}$,
            ${F_{0}\subseteq B}$ finite,
            and
            ${e,e'\in B}$.
        Our goal is to find some finite ${F\subseteq B}$ with ${F\supseteq F_{0}}$,
        such that $T \in W_{\eps;K,F,e,e'}$.

        To this end,
        we rely on the fact that $T$ is a contractive semigroup
        and observe that for all finite $F\subseteq B$
        the functions

            \begin{mathe}[mc]{rcccl}
                \tilde{f}_{F} &: &K &\to &\realsNonNeg\\
                    &&t &\mapsto &\|(\onematrix-\pi_{F})T(t)e\|,\\
                \\
                f_{F} &: &K\times K &\to &\realsNonNeg\\
                    &&(s,t) &\mapsto &|\BRAKET{(T(s)\pi_{F}T(t)-T(st))e}{e'}|\\
            \end{mathe}

        \continueparagraph
        satisfy

            \begin{mathe}[mc]{rclcccl}
                f_{F}(s,t)
                &= &|\BRAKET{(T(s)\pi_{F}T(t)-T(s)T(t))e}{e'}|\\
                &= &|\BRAKET{(\onematrix-\pi_{F})T(t)e}{T(s)^{\ast}e'}|
                &\leq &\|T(s)^{\ast}e'\|\tilde{f}_{F}(t)
                &\leq &\tilde{f}_{F}(t)\\
            \end{mathe}

        \continueparagraph
        for all $s,t\in K$.
        Furthermore, the $\topSOT$-continuity of $T$ guarantees that $\tilde{f}_{F}$ is continuous.
        Now consider the net $(\tilde{f}_{F})_{F}$,
        where the indices run over all finite $F\subseteq B$, ordered by inclusion.
        Note that the correspondingly indexed net of projections, $(\pi_{F})_{F}$, is monotone,
        and, since ${\bigcup_{F\subseteq B~\text{finite}}F=B}$ and $B$ is a basis for $\HilbertRaum$,
        it holds that ${\pi_{F}\underset{F}{\longrightarrow}\onematrix}$ weakly (in fact strongly).
        Clearly then

            \begin{mathe}[mc]{c}
            \eqtag[eq:pt-wise-convergence:\beweislabel]
                \tilde{f}_{F} \underset{F}{\longrightarrow} 0\\
            \end{mathe}

        \continueparagraph
        pointwise and monotone. Since $K$ is compact and the $\tilde{f}_{F}$ are continuous for all $F$,
        by Dini's Theorem (\cf
            \cite[Theorem~2.66]{aliprantis2005}%
        ) the monotone pointwise convergence in
            \eqcref{eq:pt-wise-convergence:\beweislabel}
        is in fact uniform convergence.
        Hence, by the definition of the net,
        for some finite $F\subseteq B$ with $F\supseteq F_{0}$

            \begin{mathe}[mc]{c}
                d_{K,F,e,e'}(T)
                    = \sup_{s,t\in K}f_{F}(s,t)
                    \leq \sup_{t\in K}\tilde{f}_{F}(t)
                    < \eps,\\
            \end{mathe}

        \continueparagraph
        and thereby ${T\in W_{\eps;K,F,e,e'}}$.

        \null

        To complete the proof of (II), we treat the second inclusion in \eqcref{eq:G-delta-expression:\beweislabel}.
        So let ${T\in\RaumX}$ be an arbitrary element in the $G_{\delta}$-set in the middle of \eqcref{eq:G-delta-expression:\beweislabel}.
        To show that ${T\in\SpHw(M)}$, it is necessary and sufficient
        to show that ${T(st)=T(s)T(t)}$ for all ${s,t\in M}$.
        So fix arbitrary ${s,t\in M}$.
        It suffices to show that
            $\BRAKET{(T(s)T(t)-T(st))e}{e'}=0$
        for all basis vectors, ${e,e'\in B}$.
        So fix arbitrary $e,e'\in B$.

        Note that since $\compactcover$ covers $M$ and is closed under finite unions,
        there exists some ${K\in\compactcover}$, such that ${s,t\in K}$.
        Fix this compact set.
        Now consider the net
            $(d_{K,F,e,e'}(T))_{F}$,
        whose indices run over all finite ${F\subseteq B}$, ordered by inclusion.
        Since $T$ is in the set in the middle of \eqcref{eq:G-delta-expression:\beweislabel},
        working through the definitions yields

            \begin{mathe}[mc]{c}
                \forall{\eps\in\rtnl^{+}:~}
                \forall{F_{0}\subseteq B~\text{finite}:~}
                    \,\exists{F\subseteq B~\text{finite},~\text{s.t.}~F\supseteq F_{0}:~}
                        \,d_{K,F,e,e'}(T)<\eps,\\
            \end{mathe}

        \continueparagraph
        which is clearly equivalent to

            \begin{mathe}[mc]{c}
                \eqtag[eq:liminf:\beweislabel]\relax
                {\displaystyle\liminf_{F}}\,d_{K,F,e,e'}(T) = 0.\\
            \end{mathe}

        Now, since $s,t\in K$, it holds that

            \begin{mathe}[mc]{c}
                |\BRAKET{(T(s)T(t)-T(st))e}{e'}|
                \leq
                    \underbrace{
                        |\BRAKET{T(s)(\onematrix-\pi_{F})T(t)e}{e'}|
                    }_{=:d'_{F}}
                    +\underbrace{
                        |\BRAKET{(T(s)\pi_{F}T(t)-T(st))e}{e'}|
                    }_{\leq d_{K,F,e,e'}(T)}\\
            \end{mathe}

        \continueparagraph
        for all finite $F\subseteq B$.
        Since ${\pi_{F}\underset{F}{\longrightarrow}\onematrix}$ weakly (see above),
        we have ${d'_{F}\underset{F}{\longrightarrow}0}$.
        Noting \eqcref{eq:liminf:\beweislabel},
        taking the limit inferior of the right hand side of the above expression
        thus yields ${\BRAKET{(T(s)T(t)-T(st))e}{e'}=0}$.
        Since ${e,e'\in B}$ were arbitrarily chosen, and $B$ is a basis for $\HilbertRaum$,
        it follows that ${T(st)=T(s)T(t)}$.
        This completes the proof.
    \end{proof}

\begin{rem}
    The proof of \Cref{thm:main-result:sig:article-str-problem-raj-dahya} reveals that in fact claims (I) and (II) hold,
    provided the topological monoid $M$ is at least $\sigma$-compact.
    And if $M$ is furthermore \usesinglequotes{good}, then these again imply that that $\SpHs(M)$ is a $G_{\delta}$-subset in $(\SpCw(M),\toplocWOT)$.
    The stronger assumption of $M$ being locally compact is only relied upon
    to obtain that $(\SpCw(M),\toplocWOT)$ is itself completely metrisable
    (\cf the proof of \Cref{prop:basic:SpC:basic-polish:sig:article-str-problem-raj-dahya}),
    and thus via Alexandroff's lemma
    that $G_{\delta}$-subsets of this space are completely metrisable.
\end{rem}

Finally, by \Cref{cor:multiparameter-auto-cts:sig:article-str-problem-raj-dahya} we obtain:

\begin{cor}
    The spaces of one-/multiparameter contractive $\C0$-semigroups on a separable Hilbert space,
    viewed under the topology of uniform $\topWOT$-convergence on compact subsets,
    are Polish.
\end{cor}

This positively solves the open problem raised in
    \cite[\S{}III.6.3]{eisner2010buchStableOpAndSemigroups}.



\setcounternach{section}{1}
\appendix



\section[Strong continuity of operator semigroups]{Strong continuity of operator semigroups}
\label{app:continuity}

\firstparagraph
The main result of this paper
is proved for semigroups defined on \usesinglequotes{good} monoids
(see \Cref{defn:good-monoids:sig:article-str-problem-raj-dahya}).
By a well-known result, any $\topWOT$-continuous semigroup
    on a Banach space $\BanachRaum$
    over the monoid $(\realsNonNeg,+,0)$
is automatically $\topSOT$-continuous,
(\cf \cite[Theorem~5.8]{Engel1999})
and thus $\realsNonNeg$ is by our definition a \usesinglequotes{good} monoid.

In this appendix, we provide sufficient conditions for topological monoids
to possess this property, and thus broaden application of the main result.
These conditions are given as follows:


\begin{defn}
\makelabel{app:continuity:defn:conditition-I:sig:article-str-problem-raj-dahya}
    A topological monoid, $M$, shall be called \highlightTerm{extendible},
    if there exists a locally compact Hausdorff topological group, $G$,
    such that $M$ is topologically and algebraically isomorphic to a closed subset of $G$.
\end{defn}

If $M$ is extendible to $G$ via the above definition,
then one can assume \withoutlog that $M \subseteq G$.

\begin{defn}
\makelabel{app:continuity:defn:conditition-II:sig:article-str-problem-raj-dahya}
    Let $G$ be a locally compact Hausdorff group.
    Call a subset $A \subseteq G$ \highlightTerm{positive in the identity},
    if for all neighbourhoods, $U \subseteq G$, of the group identity,
    $U \cap A$ has non-empty interior within $G$.
\end{defn}



\begin{e.g.}[The non-negative reals]
\makelabel{app:continuity:e.g.:extendible-mon:reals:sig:article-str-problem-raj-dahya}
    Consider ${M \colonequals \realsNonNeg}$ viewed under addition.
    Since ${M\subseteq\reals}$ is closed,
    we have that $M$ is an extendible locally compact Hausdorff monoid.
    For any open neighbourhood, ${U\subseteq\reals}$, of the identity,
    there exists an ${\eps>0}$, such that ${(-\eps,\eps)\subseteq U}$
    and thus ${U\cap M\supseteq(0,\eps)\neq\leer}$.
    Hence $M$ is positive in the identity.
\end{e.g.}

\begin{e.g.}[The $p$-adic integers]
\makelabel{app:continuity:e.g.:extendible-mon:p-adic:sig:article-str-problem-raj-dahya}
    Consider ${M \colonequals \mathbb{Z}_{p}}$ with ${p\in\mathbb{P}}$,
    viewed under addition and with the topology generated by the $p$-adic norm.
    Since ${M\subseteq\rtnl_{p}}$ is clopen,
    it is an extendible locally compact Hausdorff monoid.
    Since $M$ is clopen, it is clearly positive in the identity.
\end{e.g.}

\begin{e.g.}[Discrete cases]
\makelabel{app:continuity:e.g.:extendible-mon:discrete:sig:article-str-problem-raj-dahya}
    Let $G$ be a discrete group, and let $M\subseteq G$ contain the identity and be closed under group multiplication.
    Clearly, $M$ is a locally compact Hausdorff monoid, extendible to $G$ and positive in the identity.
    For example one can take the free-group $\mathbb{F}_{2}$ with generators $\{a,b\}$,
    and $M$ to be the algebraic closure of $\{1,a,b\}$ under multiplication.
\end{e.g.}

\begin{e.g.}[Non-discrete, non-commutative cases]
\makelabel{app:continuity:e.g.:extendible-mon:non-comm-non-discrete:sig:article-str-problem-raj-dahya}
    Let $d\in\ntrlpos$ with $d>1$ and consider the space, $\RaumX$, of $\reals$-valued $d\times d$ matrices.
    Topologised with any matrix norm (equivalently the strong or the weak operator topologies),
    this space is homeomorphic to $\reals^{d^{2}}$ and thus locally compact Hausdorff.
    Since the determinant map ${X\ni T\mapsto \det(T)\in\reals}$ is continuous,
    the subspace of invertible matrices $\{T\in\RaumX\mid \det(T)\neq 0\}$
    is open and thus a locally compact Hausdorff topological group.
    The subspace, $G$, of upper triangular matrices with positive diagonal entries,
    is a closed subgroup and thus locally compact Hausdorff.
    Letting

        \begin{mathe}[mc]{rcl}
            G_{0} &\colonequals &\{T\in G\mid \det(T)=1\},\\
            G_{+} &\colonequals &\{T\in G\mid \det(T)>1\},~\text{and}\\
            G_{-} &\colonequals &\{T\in G\mid \det(T)<1\},\\
        \end{mathe}

    \continueparagraph
    it is easy to see that $M \colonequals G_{0}\cup G_{+}$ is a topologically closed subspace containing the identity and is closed under multiplication.
    Moreover $M$ is a proper monoid, since the inverses of the elements in $G_{+}$ are clearly in ${G\ohne M}$.
    Consider now an open neighbourhood, ${U\subseteq G}$, of the identity.
    Since inversion is continuous, $U^{-1}$ is also an open neighbourhood of the identity.
    Since, as a locally compact Hausdorff space, $G$ satisfies the Baire category theorem,
    and since ${G_{+}\cup G_{-}}$ is clearly dense (and open) in $G$, and thus comeagre,
    we clearly have $(U\cap U^{-1})\cap(G_{+}\cup G_{-})\neq\leer$.
    So either ${U\cap G_{+}\neq\leer}$ or else ${U^{-1}\cap G_{-}\neq\leer}$,
    from which it follows that ${U\cap G_{+}=(U^{-1}\cap G_{-})^{-1}\neq\leer}$.
    Hence in each case ${U\cap M}$ contains a non-empty open subset, \viz ${U\cap G_{+}}$.
    So $M$ is extendible to $G$ and positive in the identity.

    Next, consider the subgroup, ${\ExampleGroupHeisenberg \subseteq G}$,
    consisting of matrices of the form $T = \onematrix + \tilde{T}$,
        where $\tilde{T}$ is a strictly upper triangular matrix
        with at most non-zero entries on the top row and right hand column.
    That is, $\ExampleGroupHeisenberg$ is the \highlightTerm{continuous Heisenberg group}, $\Heisenberg{2d-3}(\reals)$, of order ${2d-3}$.
    The elements of the Heisenberg group occur in the study of Kirillov's \highlightTerm{orbit method}
        (see \cite{kirillov1962})
    and have important applications in physics
        (see \exempli \cite{kirillov2003}).
    Clearly, $\ExampleGroupHeisenberg$ is topologically closed within $G$ and thus locally compact Hausdorff.
    Now consider the subspace,

        \begin{mathe}[mc]{c}
            \ExampleMonoidHeisenberg \colonequals \{T\in \ExampleGroupHeisenberg \mid \forall{i,j\in\{1,2,\ldots,d\}:~}T_{ij}\geq 0\},\\
        \end{mathe}

    \continueparagraph
    of matrices with only non-negative entries.
    This is clearly a topologically closed subspace of $\ExampleGroupHeisenberg$ containing the identity and closed under multiplication.
    Moreover, for each ${S,T \in \ExampleMonoidHeisenberg\ohne\{\onematrix\}}$ we have

        \begin{mathe}[mc]{c}
            ST = \onematrix + \big((S-\onematrix) + (T-\onematrix) + (S-\onematrix)(T-\onematrix)\big)
              \in \ExampleMonoidHeisenberg \ohne \{\onematrix\},\\
        \end{mathe}

    \continueparagraph
    which implies that no non-trivial element in $\ExampleMonoidHeisenberg$ has its inverse in $\ExampleMonoidHeisenberg$, making $\ExampleMonoidHeisenberg$ a proper monoid.
    Consider now an open neighbourhood, ${U \subseteq \ExampleGroupHeisenberg}$, of the identity.
    Since $\ExampleGroupHeisenberg$ is homeomorphic to $\reals^{2d-3}$,
    there exists some ${\eps>0}$, such that

        \begin{mathe}[mc]{c}
            U \supseteq \{
                T \in \ExampleGroupHeisenberg
                \mid
                \forall{(i,j)\in\mathcal{I}:~}T_{ij}\in(-\eps,\eps)
            \},\\
        \end{mathe}

    \continueparagraph
    where $\mathcal{I} \colonequals \{(1,2),(1,3),\ldots,(1,d),(2,d),\ldots,(d-1,d)\}$.
    Hence

        \begin{mathe}[mc]{rcl}
            U \cap \ExampleMonoidHeisenberg &\supseteq &\{
                T \in \ExampleGroupHeisenberg
                \mid
                \forall{(i,j)\in\mathcal{I}:~}T_{ij}\in(0,\eps)
            \} \eqcolon V,\\
        \end{mathe}

    \continueparagraph
    where $V$ is clearly a non-empty open subset of $\ExampleGroupHeisenberg$,
    since the $2d-3$ entries in the matrices can be freely and independently chosen.
    Thus $\ExampleMonoidHeisenberg$ is extendible to $\ExampleGroupHeisenberg$ and positive in the identity.

    Finally, we may consider the subgroup, ${\ExampleGroupUTunit \colonequals \UpperTrUnit(d)}$, of upper triangular matrices over $\reals$ with unit diagonal.
    The elements of $\UpperTrUnit(d)$ have important applications in image analysis
    (see \exempli
        \cite{kirillov2003}
        and
        \cite[\S{}5.5.2]{pennec2020}%
    )
    and representations of the group have been studied in
        \cite[Chapter~6]{samoilenko1991}.
    Setting ${\ExampleMonoidUT \colonequals \{T \in \ExampleGroupUTunit\mid \forall{i,j\in\{1,2,\ldots,d\}:~}T_{ij}\geq 0\}}$,
    one may argue similarly to the case of the continuous Heisenberg group
    and obtain that $\ExampleGroupUTunit$ is locally compact
    and that $\ExampleMonoidUT$ is a proper topological monoid
    which is furthermore extendible to $\ExampleGroupUTunit$ and positive in the identity.
\end{e.g.}

The following result allows us to generate infinitely many examples from basic ones:

\begin{prop}
\makelabel{app:continuity:prop:monoids-with-cond-closed-under-fprod:sig:article-str-problem-raj-dahya}
    Let ${n\in\ntrlpos}$ and let $M_{i}$ be locally compact Hausdorff monoids for ${1\leq i\leq n}$.
    Assume for each ${i<n}$ that $M_{i}$ is extendible to a locally compact Hausdorff group $G_{i}$,
    and that $M_{i}$ is positive in the identity of $G_{i}$.
    Then ${M \colonequals \prod_{i=1}^{n}M_{i}}$ is a locally compact Hausdorff monoid
    which is extendible to ${G \colonequals \prod_{i=1}^{n}G_{i}}$
    and positive in the identity.
\end{prop}

    \begin{proof}
        The extendibility of $M$ to $G$ is clear.
        Now consider an arbitrary open neighbourhood, $U$, of the identity in $G$.
        For each ${1\leq i\leq n}$, one can find open neighbourhoods, $U_{i}$, of the identity in $G_{i}$,
        so that ${U' \colonequals \prod_{i=1}^{n}U_{i}\subseteq U}$.
        By assumption, $M_{i}\cap U_{i}$ contains a non-empty open set,
        ${V_{i}\subseteq G_{i}}$ for each ${1\leq i\leq n}$.
        Since ${U\supseteq\prod_{i=1}^{n}U_{i}}$,
        it follows that
            $M\cap U
                \supseteq
                \prod_{i=1}^{n}(M_{i}\cap U_{i})
                \supseteq
                \prod_{i=1}^{n}V_{i}\neq\leer$.
        Thus ${M\cap U}$ has non-empty interior.
        Hence $M$ is positive in the identity.
    \end{proof}



Our argumentation for the generalised continuity result
builds on
    \cite[Theorem~5.8]{Engel1999}.

\begin{schattierteboxdunn}[%
    backgroundcolor=leer,%
    nobreak=true,%
]
\begin{thm}
\makelabel{app:continuity:thm:generalised-auto-continuity:sig:article-str-problem-raj-dahya}
    Let $M$ be a topological monoid
    such that
        $M$ is extendible to a locally compact Hausdorff group $G$
        and such that
        $M$ is positive in the identity.
    Then for any Banach space, $\BanachRaum$,
    every $\topWOT$-continuous semigroup, ${T:M\to\BoundedOps{\BanachRaum}}$,
    is $\topSOT$-continuous.
    In particular, $M$ is \usesinglequotes{good}.
\end{thm}
\end{schattierteboxdunn}

(Note that a semigroup over a Banach space $\BanachRaum$ on a topological monoid
is defined analogously to \Cref{defn:semigroup-on-h:sig:article-str-problem-raj-dahya}.)

    \begin{proof}
        First note that the principle of uniform boundedness applied twice to the $\topWOT$-continuous function, $T$,
        ensures that $T$ is norm-bounded on all compact subsets of $M$.
        Fix now a left-invariant Haar measure, $\lambda$, on $G$ and set

            \begin{mathe}[mc]{rcl}
                S &\colonequals &\{F\subseteq G\mid F~\text{a compact neighbourhood of the identity}\}.\\
            \end{mathe}

        Consider arbitrary ${F\in S}$ and ${x \in \BanachRaum}$.
        By the closure of $M$ in $G$ as well as positivity in the identity,
        ${M\cap F}$ is compact and contains a non-empty open subset of $G$.
        It follows that ${0<\lambda(M\cap F)<\infty}$.
        The $\topWOT$-continuity of $T$, the compactness (and thus measurability) of ${M\cap F}$,
        and the norm-boundedness of $T$ on compact subsets
        ensure that

            \begin{mathe}[mc]{rcl}
            \eqtag[eq:defn-xF:\beweislabel]
                \BRAKET{x_{F}}{\phi}
                    &\colonequals &\frac{1}{\lambda(M\cap F)}
                            \displaystyle\int_{t\in M\cap F}
                            \BRAKET{T(t)x}{\phi}
                            ~\dee{}t,
                    \quad\text{for $\phi \in \BanachRaum^{\prime}$}\\
            \end{mathe}

        \continueparagraph
        describes a well-defined element $x_{F} \in \BanachRaum^{\prime\prime}$.
        Exactly as in
            \cite[Theorem~5.8]{Engel1999},
        one may now argue by the $\topWOT$-continuity of $T$ and compactness of ${M\cap F}$
        that in fact ${x_{F} \in \BanachRaum}$ for each ${x \in \BanachRaum}$ and ${F\in S}$.
        Moreover, since $M$ is locally compact,
        and $T$ is $\topWOT$-continuous with $T(1)=\onematrix$,
        one readily obtains that each $x \in \BanachRaum$
        can be weakly approximated by the net,
            $(x_{F})_{F \in S}$,
        ordered by inverse inclusion. So

            \begin{mathe}[mc]{rcl}
                D &\colonequals &\{x_{F}\mid x \in \BanachRaum,~F \in S\}\\
            \end{mathe}

        \continueparagraph
        is weakly dense in $\BanachRaum$.
        Since the weak and strong closures of any convex subset in a Banach space coincide
        (\cf \cite[Theorem~5.98]{aliprantis2005}),
        it follows that the convex hull, $\co(D)$, is strongly dense in $\BanachRaum$.

        Now, to prove the $\topSOT$-continuity of $T$,
        we need to show that

            \begin{mathe}[mc]{rcl}
            \eqtag[eq:map:\beweislabel]
                t\in M &\mapsto &T(t)x \in \BanachRaum\\
            \end{mathe}

        \continueparagraph
        is strongly continuous for all $x \in \BanachRaum$.
        Since $M$ is locally compact and $T$ is norm-bounded on compact subsets of $M$,
        the set of $x \in \BanachRaum$ such that \eqcref{eq:map:\beweislabel} is strongly continuous,
        is itself a strongly closed convex subset of $\BanachRaum$.
        So, since $\co(D)$ is strongly dense in $\BanachRaum$,
        it suffices to prove
        the strong continuity of \eqcref{eq:map:\beweislabel}
        for each ${x\in D}$.

        To this end, fix arbitrary ${x \in \BanachRaum}$, ${F\in S}$ and ${t\in M}$.
        We need to show that ${T(t')x_{F}\longrightarrow T(t)x_{F}}$
        strongly for ${t'\in M}$ as ${t'\longrightarrow t}$.

        First recall, that by basic harmonic analysis,
        the canonical \highlightTerm{left-shift},

            \begin{mathe}[mc]{rcccl}
                L &: &G &\to &\BoundedOps{L^{1}(G)},\\
            \end{mathe}

        \continueparagraph
        defined via ${(L_{t}f)(s)=f(t^{-1}s)}$ for ${s,t\in G}$ and $f\in L^{1}(G)$,
        is an $\topSOT$-continuous morphism
        (\cf
            \cite[Proposition~3.5.6 ($\lambda_{1}$--$\lambda_{4}$)]{reiter2000}%
        ).
        Now, by compactness, ${f \colonequals \einser_{M\cap F}\in L^{1}(G)}$
        and it is easy to see that
            ${\|L_{t'}f-L_{t}f\|_{1}=\lambda(t'(M\cap F)\symmdiff t(M\cap F))}$
        for ${t'\in M}$.
        The $\topSOT$-continuity of $L$ thus yields

            \begin{mathe}[mc]{rcl}
            \eqtag[eq:1:\beweislabel]
                \lambda(t'(M\cap F) \symmdiff t(M\cap F)) &\longrightarrow &0\\
            \end{mathe}

        \continueparagraph
        for ${t'\in M}$ as ${t'\longrightarrow t}$.

        Fix now a compact neighbourhood, ${K\subseteq G}$, of $t$.
        For ${t'\in M\cap K}$ and ${\phi \in \BanachRaum^{\prime}}$ one obtains

            \begin{mathe}[mc]{rcl}
                |\BRAKET{T(t')x_{F}-T(t)x_{F}}{\phi}|
                    &= &|\BRAKET{x_{F}}{T(t')^{\ast}\phi}-\BRAKET{x_{F}}{T(t)^{\ast}\phi}|\\
                    &= &\frac{1}{\lambda(M\cap F)}
                        \cdot\left|
                            \displaystyle\int_{s\in M\cap F}\textstyle
                            \BRAKET{T(s)x}{T(t')^{\ast}\phi}~\dee{}s
                        -
                            \displaystyle\int_{s\in M\cap F}\textstyle
                            \BRAKET{T(s)x}{T(t)^{\ast}\phi}~\dee{}s
                        \right|\\
                    &\multispan{2}{\text{by construction of $x_{F}$ in \eqcref{eq:defn-xF:\beweislabel}}}\\
                    &= &\frac{1}{\lambda(M\cap F)}
                        \cdot\left|
                            \displaystyle\int_{s\in M\cap F}\textstyle
                            \BRAKET{T(t's)x}{\phi}~\dee{}s
                        -
                            \displaystyle\int_{s\in M\cap F}\textstyle
                            \BRAKET{T(ts)x}{\phi}~\dee{}s
                        \right|\\
                    &\multispan{2}{\text{since $T$ is a semigroup}}\\
                    &= &\frac{1}{\lambda(M\cap F)}
                        \cdot\left|
                            \displaystyle\int_{s\in t'(M\cap F)}\textstyle
                            \BRAKET{T(s)x}{\phi}~\dee{}s
                        -
                            \displaystyle\int_{s\in t(M\cap F)}\textstyle
                            \BRAKET{T(s)x}{\phi}~\dee{}s
                        \right|\\
                    &\multispan{2}{\text{by left-invariance}}\\
                    &\leq &\frac{1}{\lambda(M\cap F)}
                            \displaystyle\int_{s\in t'(M\cap F)\symmdiff t(M\cap F)}\textstyle
                            |\BRAKET{T(s)x}{\phi}|~\dee{}s\\
                    &\leq &\frac{1}{\lambda(M\cap F)}
                        \cdot\displaystyle\sup_{s\in(M\cap K)(M\cap F)}\textstyle\|T(s)\|
                        \cdot\|x\|\cdot\|\phi\|
                        \cdot\lambda(t'(M\cap F)\symmdiff t(M\cap F))\\
                        &\multispan{2}{\text{since $t,t'\in M\cap K$}.}\\
            \end{mathe}

        Since $K' \colonequals (M\cap K)(M\cap F)$ is compact, and $T$ is uniformly bounded on compact sets (see above),
        it holds that ${C \colonequals \displaystyle\sup_{s\in K'}\|T(s)\|\textstyle<\infty}$.
        The above calculation thus yields

            \begin{mathe}[mc]{rcl}
            \eqtag[eq:2:\beweislabel]
                \|T(t')x_{F}-T(t)x_{F}\|
                &= &\displaystyle\sup\textstyle\{
                        |\BRAKET{T(s)x_{F}-T(t)x_{F}}{\phi}|
                    \mid
                        \phi \in \BanachRaum^{\prime},~\|\phi\|\leq 1
                    \}\\
                &\leq &\frac{1}{\lambda(M\cap F)}
                        \cdot C
                        \cdot\|x\|
                        \cdot\lambda(t'(M\cap F)\symmdiff t(M\cap F))\\
            \end{mathe}

        \continueparagraph
        for all $t'\in M$ sufficiently close to $t$.

        By \eqcref{eq:1:\beweislabel}, the right-hand side of \eqcref{eq:2:\beweislabel} converges to $0$
        and hence ${T(t')x_{F}\longrightarrow T(t)x_{F}}$
        strongly as ${t'\longrightarrow t}$.
        This completes the proof.
    \end{proof}

\begin{rem}
    In the proof of \Cref{app:continuity:thm:generalised-auto-continuity:sig:article-str-problem-raj-dahya},
    weak continuity only played a role in obtaining
    the boundedness of $T$ on compact sets,
    as well as the well-definedness of the elements in $D$.
    In \cite[Theorem~9.3.1 and Theorem~10.2.1--3]{hillephillips2008} a proof of the classical continuity result exists under weaker conditions,
    \viz weak measurability, provided the semigroups are almost separably valued.
    It would be interesting to know whether the above approach can be adapted to these weaker assumptions.
\end{rem}







\subsection*{Acknowledgement}

\firstparagraph
The author is grateful
    to Tanja Eisner for her feedback,
    to Konrad Zimmermann for his helpful comments on the results in the appendix,
and
    to the referee for their constructive feedback.


\bibliographystyle{abbrv}

\footnotesize
\begin{thebibliography}{10}

\bibitem{aliprantis2005}
C.~D. Aliprantis and K.~C. Border.
\newblock {\em {Infinite Dimensional Analysis, a Hitchhiker's Guide}}.
\newblock Springer-Verlag, 3rd edition, 2005.

\bibitem{dahya2022weakproblem}
R.~Dahya.
\newblock {The space of contractive $C_{0}$-semigroups is a Baire space}.
\newblock {\em Journal of Mathematical Analysis and Applications},
  508(1):125881, 2022.
\newblock Confirmed for publication.

\bibitem{eisner2010typicalContraction}
T.~Eisner.
\newblock {A ‘typical’ contraction is unitary}.
\newblock {\em L'Enseignement Math{\'e}matique}, 56(2):403--410, 2010.

\bibitem{eisner2010buchStableOpAndSemigroups}
T.~Eisner.
\newblock {\em {Stability of Operators and Operator Semigroups}}, volume 209 of
  {\em Operator Theory: Advances and Applications}.
\newblock Birkhäuser, Basel, 2010.

\bibitem{eisnermaitrai2010typicalOperators}
T.~Eisner and T.~M{\'a}trai.
\newblock {On typical properties of Hilbert space operators}.
\newblock {\em Israel Journal of Mathematics}, 195, 2010.

\bibitem{eisnersereny2009catThmStableSemigroups}
T.~Eisner and A.~Ser{\'e}ny.
\newblock {Category theorems for stable semigroups}.
\newblock {\em Ergodic Theory and Dynamical Systems}, 29(2):487--494, 2009.

\bibitem{Eisner2008kato}
T.~Eisner and A.~Ser{\'e}ny.
\newblock {On the weak analogue of the Trotter-Kato theorem}.
\newblock {\em Taiwanese Journal of Mathematics}, 14(4):1411--1416, 2010.

\bibitem{Engel1999}
K.-J. Engel and R.~Nagel.
\newblock {\em {One-Parameter Semigroups for Linear Evolution Equations}}.
\newblock Graduate Texts in Mathematics. Springer-Verlag, 1999.

\bibitem{grivaux2021localspecLp}
S.~Grivaux and {\'E}.~Matheron.
\newblock {Local spectral properties of typical contractions on
  $\ell_{p}$-spaces}.
\newblock Preprint available under \url{https://arxiv.org/abs/2105.04635},
  2021.

\bibitem{grivaux2021invariantsubspaceLp}
S.~Grivaux, E.~Matheron, and Q.~Menet.
\newblock {Does a typical {$\ell_{p}$-space} contraction have a non-trivial
  invariant subspace?}
\newblock {\em Trans. Amer. Math. Soc.}, 374(10):7359--7410, 2021.

\bibitem{grivaux2021typicalexamples}
S.~Grivaux, {\'E}.~Matheron, and Q.~Menet.
\newblock {Linear dynamical systems on {H}ilbert spaces: typical properties and
  explicit examples}.
\newblock {\em Mem. Amer. Math. Soc.}, 269(1315):v+147, 2021.

\bibitem{hillephillips2008}
E.~Hille and R.~S. Phillips.
\newblock {\em {Functional analysis and semi-groups}}, volume~31.
\newblock American Mathematical Society, 2008.

\bibitem{yoshi1949}
{H}isaaki {Y}oshizawa.
\newblock {Unitary representations of locally compact groups: Reproduction of
  Gelfand-Raikov's theorem}.
\newblock {\em Osaka {M}athematical {J}ournal}, 1:81--89, 1949/2003.

\bibitem{kech1994}
A.~S. Kechris.
\newblock {\em {Classical Descriptive Set Theory}}.
\newblock Springer-Verlag, 1994.

\bibitem{kirillov1962}
A.~A. Kirillov.
\newblock {Unitary representations of nilpotent Lie groups}.
\newblock {\em Russian Mathematical Surveys}, 17(4):53--104, 1962.

\bibitem{kirillov2003}
A.~A. Kirillov.
\newblock Two more variations on the triangular theme.
\newblock In {\em The Orbit Method in Geometry and Physics}, pages 243--258.
  Birkh\"auser Boston, 2003.

\bibitem{krol2009}
S.~Kr{\'o}l.
\newblock {A note on approximation of semigroups of contractions on Hilbert
  spaces}.
\newblock {\em Semigroup Forum}, 79(2):369--376, 2009.

\bibitem{lopushanskyMultiparamFourier}
O.~Lopushansky and S.~Sharyn.
\newblock {Operators commuting with multi-parameter shift semigroups}.
\newblock {\em Carpathian Journal of Mathematics}, 30(2):217--224, 2014.

\bibitem{murphy1990}
G.~J. Murphy.
\newblock {\em {C\textsuperscript{$\ast$}-Algebras and Operator Theory}}.
\newblock Academic Press, San Diego, 1990.

\bibitem{peller1981estimatesOperatorPolyLp}
V.~V. Peller.
\newblock {Estimates of operator polynomials in the space ${L}^{p}$ with
  respect to the multiplier norm}.
\newblock {\em Journal of Soviet Mathematics}, 16(3):1139--1149, 1981.

\bibitem{pennec2020}
X.~Pennec and M.~Lorenzi.
\newblock {Beyond Riemannian geometry: The affine connection setting for
  transformation groups}.
\newblock In X.~Pennec, S.~Sommer, and T.~Fletcher, editors, {\em {Riemannian
  Geometric Statistics in Medical Image Analysis}}, pages 169--229. Academic
  Press, 2020.

\bibitem{reiter2000}
H.~Reiter and J.~D. Stegeman.
\newblock {\em {Classical Harmonic Analysis and Locally Compact Groups}}.
\newblock Oxford University Press, 2nd edition, 2000.

\bibitem{samoilenko1991}
Y.~S. Samoilenko.
\newblock {\em {Spectral Theory of Families of Self-Adjoint Operators}}, pages
  {124--144}.
\newblock {Springer Netherlands}, {Dordrecht}, 1991.

\bibitem{terehin1975}
A.~P. Terehin.
\newblock {A multiparameter semigroup of operators, mixed moduli and
  approximation}.
\newblock {\em Mathematics of the {USSR}-Izvestiya}, 9(4):887--910, 1975.

\bibitem{zelik2004}
S.~Zelik.
\newblock {Multiparameter semigroups and attractors of reaction-diffusion
  equations in $\mathbb{R}^{n}$}.
\newblock {\em Transactions of the Moscow Mathematical Society}, 65:105--160,
  2004.

\end{thebibliography}
\def\bibname{References}
\bgroup

\egroup


\addresseshere
\end{document}
